\documentclass[10pt,reqno,a4paper]{amsart}

\usepackage{mdframed}

\usepackage[latin1]{inputenc}
\usepackage{tikz}
\usetikzlibrary{shapes,arrows}
\usetikzlibrary{arrows.meta}
\usepackage{forest}
\usepackage{tikz-qtree}

\usetikzlibrary{arrows,shapes,positioning,shadows,trees}

\usepackage[linkcolor=red,colorlinks=true]{hyperref}

\usepackage{etoolbox}
\usepackage{amsmath}
\usepackage{enumerate}
\usepackage{amssymb}
\usepackage{amscd}
\usepackage{amsthm}
\usepackage{amsfonts}
\usepackage{graphicx}
\usepackage[all,cmtip]{xy}
\usepackage{enumitem}

\patchcmd{\subsection}{-.5em}{.5em}{}{}
\patchcmd{\subsubsection}{-.5em}{.5em}{}{}

\usepackage{enumitem}

\makeatother
\usepackage{hyperref}

\numberwithin{equation}{section}

\newcommand{\SL}{\operatorname{SL}}

\newcommand{\re}{\operatorname{Re}}

\newcommand{\cA}{\mathcal{A}}

\newcommand{\cC}{\mathcal{C}}

\newcommand{\bC}{\mathbb{C}}

\newcommand{\bN}{\mathbb{N}}

\newcommand{\bR}{\mathbb{R}}

\newcommand{\gou}{\mathfrak{u}}

\newcommand{\ra}{\rightarrow}

\newcommand{\qand}{\quad \textrm{and} \quad}

\newcommand\subsetsim{\mathrel{%
\ooalign{\raise0.2ex\hbox{$\subset$}\cr\hidewidth\raise-0.8ex\hbox{\scalebox{0.9}{$\sim$}}\hidewidth\cr}}}
\newcommand{\eps}{\varepsilon}

\DeclareMathOperator{\Hom}{Hom}

\DeclareMathOperator{\Ad}{Ad}

\newcommand{\bt}{\textbf{t}}

\definecolor{lichtgrijs}{gray}{0.95}
\mdfdefinestyle{mystyle}{ %
    backgroundcolor=lichtgrijs, %
    linewidth=0pt, %
    innertopmargin=10pt, %
    innerbottommargin=10pt,%
    nobreak=true
}

\newmdtheoremenv[style=mystyle]{theorem}{Theorem}[section]
\newtheorem{corollary}[theorem]{Corollary}
\newmdtheoremenv[style=mystyle]{proposition}[theorem]{Proposition}
\newmdtheoremenv[style=mystyle]{lemma}[theorem]{Lemma}

\theoremstyle{definition}

\tikzstyle{decision} = [diamond, draw, fill=blue!20, 
    text width=4.5em, text badly centered, node distance=3cm, inner sep=0pt]
\tikzstyle{block} = [rectangle, draw, fill=blue!20, 
    text width=5em, text centered, rounded corners, minimum height=2em]
\tikzstyle{line} = [draw, -latex']
\tikzstyle{cloud} = [draw, ellipse,fill=red!20, node distance=3cm,
    minimum height=2em]

\renewcommand\labelenumi{(\roman{enumi})}
\renewcommand\theenumi\labelenumi

\usepackage{changepage}
\usepackage{mathtools}
\usepackage{graphicx}
\usepackage{calc}

\usepackage{setspace}
\fontsize{10pt}{13pt}\selectfont
\begin{document}
\bibliographystyle{plain} 

\title[Effective multiple equidistribution]{Effective multiple equidistribution of translated measures}

\author{Michael Bj\"orklund}
\address{Department of Mathematics, Chalmers, Gothenburg, Sweden}
\email{micbjo@chalmers.se}

\author{Alexander Gorodnik}
\address{Department of Mathematics, University of Z\"urich, Switzerland}
\email{alexander.gorodnik@math.uzh.ch}


\date{}

\begin{abstract}
We study the joint distributions of translated measures supported on leaves which are expanded by subgroups of diagonal matrices and generalize previous results of Kleinbock--Margulis, Dabbs--Kelly--Li, and Shi.  More specifically,   we establish  quantitative estimates on higher-order correlations for measures with low regularities and derive error terms which only depend on the distances between translations.
\end{abstract}

\maketitle

\section{Introduction}

Let us consider a jointly continuous action of a topological group $G$ on a topological space $X$.  Given a finite Borel measure $\sigma$ on $X$,  one is often interested in the asymptotic behaviour of integrals of the form
$$
\sigma(\phi\circ g):=\int_X \phi(gx)\,d\sigma(x), \quad \hbox{with $\phi\in C_b(X)$,}
$$ 
as $g\to\infty$ in $G$,  as well as in the asymptotic behaviour of integrals of the form
$$
\sigma\big(\phi_1\circ g_1 \cdots \phi_r\circ g_r\big):=\int_X \phi_1(g_1x)\cdots \phi_r(g_rx)\,d\sigma(x), \quad \hbox{with $\phi_1,\ldots,\phi_r\in C_b(X)$,}
$$ 
as $g_i\to\infty$ and $g_ig_j^{-1}\to\infty$ for $i\ne j$ in $G$.  The latter integrals are often called \emph{$r$-correlations}.  In this paper we will be interested in quantitative estimates on $r$-correlations for actions on homogeneous spaces.  

Basic models for the type of results that we are after can be found in  \cite{Sar,Zag} which deals with equidistribution of closed horocycles on the modular surface.  
More generally,  one can consider the action of a partially hyperbolic
one-parameter subgroup $(g_t)$ in $\hbox{SL}_d(\mathbb{R})$ on the homogeneous space $X:= \hbox{SL}_d(\mathbb{R})/\hbox{SL}_d(\mathbb{Z})$, equipped with the (unique) $\SL_d(\bR)$-invariant probability invariant measure $\mu$.  Kleinbock and Margulis \cite{KM1} showed that in this case there exists $\delta>0$ such that for every compactly supported smooth probability measure $\sigma$ on an unstable leaf of $(g_t)$ and for every $\phi\in C^\infty_c(X)$,  there is a constant $C(\sigma,\phi)$,  such that
\begin{equation}
\label{eq:km0}
\big|\sigma(\phi\circ g_t)-\mu(\phi)\big|\le C(\sigma,\phi)\,e^{-\delta t}
\end{equation}
for all $ t \geq 0$.  In fact, the result of \cite{KM1} applies more generally to partially hyperbolic flows 
on homogeneous spaces of semisimple Lie groups.  Subsequently, Dolgopyat \cite{Dol} developed an inductive argument which allows one to deduce from estimates of the form \eqref{eq:km0},
quantitative estimates on higher-order correlations.  In the above setting described above,  this argument tells us that there exists $\delta'>0$ such that 
for every $t_1,\ldots,t_r\ge 0$
$$
\big|\sigma\big(\phi_1\circ g_{t_1}\cdots \phi_r\circ g_{t_r}\big)-\mu(\phi_1)\cdots\mu(\phi_r)\big|\le C(\sigma,\phi_1,\ldots, \phi_r)\,e^{-\delta' D_r(t_1,\ldots,t_r)}
$$
for all $\phi_1,\ldots,\phi_r\in C^\infty_c(X)$,
where 
$$
D_r(t_1,\ldots,t_r):=\min(t_1, t_2-t_1,\ldots, t_r-t_{r-1}).
$$
The idea behind this argument in \cite{KM1} goes back to Margulis' thesis \cite{Mar},
and uses that the flow $(g_t)$ is non-expanding transversally to the unstable leaves,
so that one can "thicken" the measure $\sigma$ to reduce the original problem to mixing estimates for the flow $(g_t)$ with respect to the volume measure $\mu$.

The problem becomes significantly harder when the measure $\sigma$ is supported on a proper submanifold of the unstable leaf.  A particular instance of this problem was investigated by Kleinbock and Margulis in \cite{KM2}.  They consider the case when $\sigma$ is a smooth measure,  compactly supported on an orbit of the subgroup   
\begin{equation}
\label{eq:U}
U_{m,n}:=\left\{\left(\begin{tabular}{cc}
$I_m$	& $u$\\
$0$ & $I_n$
\end{tabular}\right):\, u\in \hbox{Mat}_{m\times n}(\bR)	
	 \right\},
\end{equation}
which is translated by the multi-parameter flow
$$
g_{\bt}:=\hbox{diag}\big(e^{t_1},\ldots,e^{t_m},e^{-t_{m+1}},\ldots,e^{-t_{m+n}}\big) 
$$
with $\sum_{i=1}^{m} t_i=\sum_{j=m+1}^{m+n} t_j$.
In this case, the orbits of $U_{m,n}$ are unstable manifolds of $g_{\bt}$ only when $t_1=\cdots =t_m\to+\infty$ and $t_{m+1}=\cdots =t_{m+n}\to+\infty$, but not in general.
The main result in \cite{KM2} states that there exists 
$\delta>0$,  which is independent of $\sigma$,  such that
\begin{equation}
	\label{eq:km}
	\big|\sigma(\phi\circ g_{\bt})-\mu(\phi)\big|\le C(\sigma,\phi)\,e^{-\delta \lfloor \bt\rfloor }
\end{equation}
for all $\phi\in C^\infty_c(X)$,
where  
$$
\lfloor \bt\rfloor:=\min(t_1,\ldots,t_{m+n}).
$$ 
This result was generalized to homogeneous spaces of semisimple groups by Dabbs, Kelly, Li \cite{DKL} and by Shi \cite{Shi}.  Quantitative estimates on higher order correlations were also established in \cite{Shi}: for every integer $r \geq 1$,  there exists $\delta' = \delta'(r) >0$ such that 
\begin{equation}
\label{eq:mult}
\big|\sigma\big(\phi_1\circ g_{\bt_1}\cdots \phi_r\circ g_{\bt_r}\big)-\mu(\phi_1)\cdots\mu(\phi_r)\big|\le C(\sigma,\phi_1,\ldots, \phi_r)\,e^{-\delta' D_r(\bt_1,\ldots,\bt_r)}
\end{equation}
for all $\phi_1,\ldots,\phi_r\in C^\infty_c(X)$,
where 
$$
D_r(\bt_1,\ldots,\bt_r):=\min\big(\lfloor \bt_1\rfloor, \lfloor \bt_2-\bt_1\rfloor,\ldots,
\lfloor \bt_r-\bt_{r-1}\rfloor\big).
$$
We stress that this estimate only provides non-trivial information when the vectors $\bt_1,\ldots,\bt_r$ are \emph{completely ordered} with respect to the function $\lfloor\cdot \rfloor$ and \emph{all} gaps with respect to this order go to infinity.
This condition is too restrictive for some of the applications that we have in mind.  

In this paper,  we generalize this result in two ways.
Firstly,  we show that one can reduce the regularity assumptions on the measure $\sigma$.  Secondly,  we establish a favourable estimate which only depend on $\lfloor \bt_1 \rfloor, \ldots, \lfloor \bt_1 \rfloor$ and the pairwise Euclidean distances $\| \bt_i-\bt_j \|$
for $i\ne j$.  

To formulate our first main result,  we need some notation.
Let $Y$ be a compact orbit of $U_{m,n}$ in $X$. Then $Y$ can be considered as a $mn$-dimensional torus, and we denote by $m_Y$ the probability invariant measure on $Y$.
We also write $\widehat{Y}$ for the character group of $Y$. Given a Borel measure $\sigma$ on $Y$, we write $\hat{\sigma}(\chi):=\int_Y \chi\, d\sigma$, $\chi\in \widehat{Y}$, for the corresponding Fourier coefficients. 
We say that $\sigma$ is a \emph{Wiener measure} if 
$$
\|\sigma\|_W:={\sum}_{\chi\in \widehat{Y}} |\hat \sigma(\chi)|<\infty.
$$
Note that every Wiener measure on $Y$ is absolutely continuous with respect to $m_Y$,  with a continuous (but possibly nowhere differentiable) density.  
We denote by $S_d$ the family of norms introduced in Section \ref{sec:sobolev} below.
In particular, one can take them to be the $C^k$-norm on $\cC^\infty_c(X)$ for a fixed sufficiently large $k$.

\begin{theorem}\label{th:special}
For every $r\ge 1$, there exist $d_r\in\bN$ and $C_r,\delta_r>0$ such that
for every Wiener probability measure $\sigma$ supported on a compact orbit of $U_{m,n}$ in $X$, $\bt_1,\ldots,\bt_r\in \bR_+^{m+n}$, and $\phi_1,\ldots,\phi_r\in C_c^\infty(X)$, 
$$
\big|\sigma\big(\phi_1\circ g_{\bt_1}\cdots \phi_r\circ g_{\bt_r}\big)-\mu(\phi_1)\cdots\mu(\phi_r)\big|\le C_r\, \|\sigma\|_W S_{d_r}(\phi_1)\cdots S_{d_r}(\phi_r)\,e^{-\delta_r \Delta_r(\bt_1,\ldots,\bt_r)},
$$
where 
$$
\Delta_r(\bt_1,\ldots,\bt_r):=\min\big(\lfloor \bt_i\rfloor, \| \bt_i-\bt_j \| :\, 1\le i\ne j\le r\big).
$$
\end{theorem}

\vspace{0.3cm}

Let us now formulate a more general version of this result. 
Let $G$ be a connected semisimple Lie group without compact factors
and $P$ a parabolic subgroup of $G$ such that the projection of $P$ to every simple factor of $G$ is proper. Let $U$ be the unipotent radical of $P$ and $A$ a maximal connected $\hbox{Ad}$-diagonalizable subgroup of $P$. We write $\mathfrak{a}$ and 
$\mathfrak{u}$ for the corresponding Lie algebras and
consider the adjoint action of $\mathfrak{a}$ on $\mathfrak{u}$.  Fix a norm on $\mathfrak{a}$ and let $d$ denote the corresponding invariant distance function on $A$.
Let $\Phi(\mathfrak{u})$ denote the set of roots (characters of $\mathfrak{a}$)
arising in this action. Let $\mathfrak{a}_U^{*,+}$ be the positive cone
in the dual $\mathfrak{a}^*$ spanned by the characters in $\Phi(\mathfrak{u})$.
We use the usual identification $\mathfrak{a}^*\to \mathfrak{a}$ defined by the Killing form: $\alpha\mapsto s_\alpha$ given by $\left< s_\alpha, a\right>=\alpha(a)$ for $a\in \mathfrak{a}$, and denote by $\mathfrak{a}_U^{+}$ the corresponding cone in $\mathfrak{a}$ given by this identification. The cone
$A_U^+:=\exp(\mathfrak{a}_U^{+})$
was introduced in \cite{Shi} and is called the expanding cone for $U$.
For $a\in A_U^+$, we denote by $\lfloor a\rfloor_U$ the distance from $a$ to the boundary of the cone $A_U^+$. 
Let $\Gamma$ be an irreducible lattice in $G$ and $X:=G/\Gamma$ 
equipped with the invariant probability measure $\mu$.

The main result of \cite{Shi} is a generalization of the estimate \eqref{eq:mult} to compactly supported smooth  measures on an orbit of $U$ in $X$. 
Let us additionally assume that $U$ is abelian and $Y$ is a compact orbit of $U$ in $X$. Given a Borel measure $\sigma$ on $U$,  we define the Wiener norm as above.  In this setting, we establish the following general version of Theorem \ref{th:special}: \\

\begin{theorem}\label{th:general}
For every $r\ge 1$, there exist $d_r\in \bN$ and  $C_r,\delta_r>0$ such that for 
every Wiener probability measure $\sigma$ supported on $Y$,
$t_1,\ldots,t_r\in A_U^+$, and $\phi_1,\ldots,\phi_r\in C_c^\infty(X)$, 
	$$
	\big|\sigma\big(\phi_1\circ t_1\cdots \phi_r\circ t_r\big)-\mu(\phi_1)\cdots\mu(\phi_r)\big|\le C_r\, \|\sigma\|_W S_{d_r}(\phi_1)\cdots S_{d_r}(\phi_r)\,e^{-\delta_r \Delta_r(t_1,\ldots,t_r)},
	$$
	where 
	$$
	\Delta_r(t_1,\ldots,t_r):=\min\big(\lfloor t_i\rfloor_U,\, d(t_i,t_j) :\, 1\le i\ne j\le r\big).
	$$
\end{theorem}

\vspace{0.2cm}

As we clarify below, our method does not rely on particular properties of the horospherical subgroup $U$ and allows us to produce quantitative estimates on  $\sigma(\phi\circ t_1\cdots \phi\circ t_r)$ once estimates on $\sigma(\phi\circ t)$ are available (see Theorem \ref{ThmA}).

\subsection{Acknowledgements}
M.B. was supported by the Swedish Research Council grant 2019-04774 , and A.G. was supported by SNF grant 200021--182089.

\section{General result}\label{sec:general}
Let $G$ be a real Lie group, $U$ a closed connected subgroup of $G$, and $T$ a closed connected abelian subgroup of $G$ that normalizes $U$ such that the adjoint action of $T$ on the Lie algebra of $U$ is proper and diagonalizable. 
Let $X$ be a standard Borel space equipped with measurable  action of $G$.
Let $\mu$ be a $G$-invariant probability measure on $X$, and let $\nu$ be a $U$-invariant and $U$-ergodic measure on $X$.  We assume that the measure $\nu$ has discrete spectrum. This means that there is an orthonormal basis of $L^2(\nu)$ consisting of $U$-eigenfunctions,  and allows us to introduce a Wiener norm on $L^\infty(\nu)$ (see Section \ref{sec:wiener} below).  

We shall further assume that certain equidistribution properties hold for $\nu,  T$ and a sub-algebra $\cA$ of bounded functions on $X$,  which is equipped with a family of norms $S_d$ (see Section \ref{sec:sobolev} below).
For example,  $\cA$ could be $\bC + C_c^\infty(X)$ and $S_d$ could be the 
$C^k$-norms for a fixed sufficiently large $k$.

Finally,  we make the following two assumptions regarding quantitative equidistribution on $X$ in terms of a fixed suitable norm $S_{d_o}$:
\vspace{0.1cm}
\begin{itemize}
\item There exist $T_+ \subset T$, $D_o \geq 1$, $\delta_o \in (0,1)$,
 and a function 
$\rho : T_{+} \ra [1,\infty)$
such that
\begin{equation}\label{eq:eq1}\tag{EQ1}
\big| \nu(\varphi \circ t) - \mu(\varphi) \big| \leq D_o\, \rho(t)^{-\delta_o}\, S_{d_o}(\varphi), 
\end{equation}
for all $\varphi \in \cA$ and $t \in T_{+}$.   \vspace{0.2cm}

\item There exist $C \geq 1$, and $0 < c < 1/2$ such that
\begin{equation}\label{eq:eq2}\tag{EQ2}
\big| \mu(\varphi_1 \circ \exp(w) \cdot {\varphi}_2) - \mu(\varphi_1) {\mu(\varphi_2)} \, \big| \leq C \max\big(1,\|w\|\big)^{-c} \, S_{d_o}(\varphi_1) \, S_{d_o}(\varphi_2)
\end{equation}
for all $\varphi_1, \varphi_2 \in \cA$ and $w \in \hbox{Lie}(U)$. 
\end{itemize}

\noindent We fix an invariant metric $d$ on $T$.
Given $r \geq 2$ and $\underline{t} = (t_1,\ldots,t_r) \in T_{+}^r$, we define
\begin{equation}
\label{defrhormr}
\rho_r(\underline{t}): = \min\big(\rho(t_1),\ldots,\rho(t_r)\big)
\qand
m_r(\underline{t}): = \min_{i \neq j} \exp(d(t_i, t_j)),
\end{equation}
and for $r \geq 1$, we set
\begin{equation}
\label{Deltar}
\Delta_r(\underline{t}): = 
\left\{ 
\begin{array}{cl}
\rho(t_1) & \textrm{if $r = 1$} \\
\min\big(\rho_r(\underline{t}),m_r(\underline{t})\big) & \textrm{if $r \geq 2$}.
\end{array}
\right.
\end{equation}
Later on in our argument,  we shall slightly modify the definition of $m_r$ (see \eqref{eq:m_r} below). \\

Our main theorem provides quantitative estimates on $r$-correlations in terms of $\Delta_r$: \\

\begin{theorem}
	\label{ThmA}
	For every $r \geq 1$, there exist
	$d_r\in\bN$ and $D_r,\delta_r > 0$ such that
	\[
	\Big| \nu\Big(\varphi_o \prod_{i=1}^r \varphi_i \circ t_i\Big) - \nu(\varphi_o) \prod_{i=1}^r \mu(\varphi_i) \Big| \leq 
	D_r \Delta_r(\underline{t})^{-\delta_r} \|\varphi_o\|_{W(\nu)} \prod_{i=1}^r S_{d_r}(\varphi_i),
	\]
	for all $\underline{t} = (t_1,\ldots,t_r) \in T_{+}^r$, $\varphi_o \in W(\nu)$, and $\varphi_1,\ldots,\varphi_r \in \cA$.
\end{theorem}

\medskip

We note that Theorem \ref{th:general} follows immediately from Theorem \ref{ThmA}.
Indeed, every Wiener measure on a torus has a continous density,
and the assumptions \eqref{eq:eq1} and \eqref{eq:eq2} can been verified 
in this setting. 
In particular, \eqref{eq:eq1} was established in \cite{Shi}, and \eqref{eq:eq2} is the well-known exponential mixing estimate (see, for instance, \cite[2.4.3]{KM1}).

\smallskip

In Section \ref{sec:explicit} below, we also work out the parameters $d_r,D_r,\delta_r$ in Theorem \ref{ThmA} explicitly.

\section{Notations}
\label{sec:notation}

\subsection{The Wiener algebra}\label{sec:wiener}

We recall that the measure $\nu$ is assumed to have discrete spectrum with respect to $U$, that is, $L^2(\nu)$ has an orthonormal basis consisting of $U$-eigenfunctions: there exists a set $\Xi$ of unitary characters on $\hbox{Lie}(U)$ and an orthonormal basis $\{ \psi_\xi \, : \, \xi \in \Xi \}$ of $L^2(\nu)$ such that 
\begin{equation}
\label{psi_xi_is_an_eigenfunction}
\psi_\xi \circ \exp(w) = \xi(w) \psi_\xi, \quad \textrm{for all $\xi \in \Xi$ and $w \in \hbox{Lie}(U)$},
\end{equation}
where the identity is interpreted in the $L^2(\nu)$-sense. Without loss of generality, $\psi_1 = 1$. Furthermore, 
for every $\xi \in \Xi$, 
\begin{equation}
\label{psi_xi_is_bounded}
|\psi_\xi| = 1, \quad \textrm{$\nu$-almost everywhere},
\end{equation}
by the $U$-ergodicity of $\nu$.

\smallskip

We denote by $B(X)$ the $*$-algebra of bounded complex-valued measurable functions on $X$ under pointwise multiplication. 
The supremum norm on $B(X)$ is denoted by $\|\cdot\|_\infty$. 
For $\varphi \in B(X)$, we set
\begin{equation}
\label{Def_Wnorm}
\|\varphi\|_{W(\nu)} := \sum_{\xi \in \Xi} |\nu(\varphi \cdot \overline{\psi}_\xi)|,
\end{equation}
where $\big\{ \psi_\xi : \xi \in \Xi\big\}$ is a fixed orthonormal basis of $L^2(\nu)$. 

\smallskip

We define the \emph{Wiener algebra} 
\[
W(\nu) := \Big\{ \varphi \in B(X) \, : \, \|\varphi\|_{W(\nu)} < \infty \Big\}.
\]
For $\varphi \in W(\nu)$,
\begin{equation}
\label{varphiexpansion}
\varphi = \sum_{\xi \in \Xi} \nu(\varphi \cdot \overline{\psi}_\xi) \psi_\xi, 
\end{equation}
with convergence in the $L^\infty(\nu)$-sense,  In particular, 
\begin{equation}
\label{W_is_uniform}
\|\varphi\|_{L^\infty(\nu)} \leq \|\varphi\|_{W(\nu)}, \quad \textrm{for all $\varphi \in W(\nu)$}.
\end{equation}

\subsection{A family of norms on the space $X$}\label{sec:sobolev}

Let $\cA$ be a $G$-invariant sub-$*$-algebra of the algebra $B(X)$. We assume that $\cA$ is equipped with an increasing family of norms 
$$
S_1\le S_2\le \cdots \le S_d\le \cdots
$$
satisfying the following assumptions for fixed $d_o$ and for all $d$:
\vspace{0.1cm}
\begin{itemize}
\item For all $\varphi \in \cA$,
\begin{equation}
\label{S_is_uniform}\tag{S1}
\|\varphi\|_\infty \leq S_{d_o}(\varphi).
\end{equation}
\vspace{0.1cm}
\item There exist $A \geq 1$ and $a > 0$ such that
\begin{equation}
\label{S_is_uniform2}\tag{S2}
\big\|\varphi \circ \exp(w) - \varphi\big\|_\infty \leq A\, \|w\|^a\, S_{d_o}(\varphi)\quad\hbox{for all $\varphi \in \cA$ and $w \in \hbox{Lie}(U)$.}
\end{equation}
\vspace{0.1cm}

\item There exist $B_d \geq 1$ and $b_d > \max(1/2,a/4)$ such that
\begin{equation}
\label{S_is_uniform3}\tag{S3}
S_d\big(\varphi \circ \exp(w)\big) \leq B_d \, \max\big(1,\|w\|\big)^{b_d} \, S_d(\varphi)\quad
\hbox{for all $\varphi \in \cA$ and $w \in \hbox{Lie}(U)$.}
\end{equation}
\vspace{0.1cm}

\item There exists $M_{d} \geq 1$ such that
\begin{equation}
\label{S_is_submultiplicative}\tag{S4}
S_d(\varphi_1 \varphi_2) \leq M_{d}\, S_{d+d_o}(\varphi_1) S_{d+d_o}(\varphi_2), \quad \textrm{for all $\varphi_1, \varphi_2 \in \cA$}.
\end{equation}
\end{itemize}
\vspace{0.1cm}

\noindent Since $\cA$ is a $G$-invariant subalgebra,  for 
$\varphi_1,\ldots,\varphi_r \in \cA$ and $g_1,\ldots,g_r\in G$,  we also have
that the product $\prod_{i=1}^r \varphi_i \circ g_i\in \cA$.
We further assume that the norms $S_d$ are convex in the following strong sense:
\begin{itemize}
\item For every $\varphi_1,\ldots,\varphi_r \in \cA$, $g_1,\ldots,g_r\in G$, 
and compactly supported complex finite measures $\omega$ on $G^r$,
\begin{equation}
\label{S_is_convex0}\tag{S5}
\phi := \int_{G^r} \left(\prod_{i=1}^r \varphi_i \circ g_i\right) \, d\omega(g_1,\ldots,g_r)\in \cA,
\end{equation}
and 
\begin{equation}
\label{S_is_convex}\tag{S6}
S_d(\phi) \leq \int_{G^r} S_d\left(\prod_{i=1}^r \varphi_i \circ g_i\right) \, d\omega(g_1,\ldots,g_r).
\end{equation}
\end{itemize}

\medskip

We now provide several examples of norms for which our set-up applies
when $X$ is a finite-volume homogeneous space of a Lie group $G$
and $\cA=\bC+C_c^\infty(X)$:

\begin{enumerate}
	\item  The simplest example to which our framework applies is a fixed $C^k$-norm with $k\ge 1$:
	$$
	S(\phi):=\max \|\phi\|_{C^k}.
	$$
	In this case, there is no dependence on index$d$. 
	We note that here an in the other examples properties (S5)--(S6) can be  verified using the Dominated Convergence Theorem.  \vspace{0.2cm}
	
	\item For some applications, one is required to approximate unbounded functions on $X$. Then the following refined norms are useful.
	Let $\|\cdot\|_{Lip}$ denote the Lipschitz norm $\cA$
	with respect to invariant Riemannian metric
	and $\|\cdot\|_{L^p_k}$ denote 	the $L^p$-Sobolev norm of order $k$.
	One can take the family of norm
	\begin{equation}
	\label{eq:S000}
	S_d(\phi):=\max \left(\|\phi\|_{C^0},\|\phi\|_{Lip}, \|\phi\|_{L^{2^d}_k}\right).
	\end{equation}
	In this case, property \eqref{S_is_uniform3} holds with fixed $B_d,b_d$ depending only $k$, and property  \eqref{S_is_submultiplicative} with $M_d$ depending only on $k$ and with $d_o=1$ follows from the Cauchy--Schwarz inequality.  \vspace{0.2cm}

	\item The third example is the Sobolev norms uses in \cite{BEG} 
	(see \cite[Subsection 2.2]{BEG}). In this case, $d$ denotes the degree of the Sobolev
	norm, property \eqref{S_is_uniform3} holds with 
	$B_d=L_1^d$ and $b_d=\ell d$ for some $L_1,\ell \ge 1$,
	and property \eqref{S_is_submultiplicative} holds with $M_d=L_2^d$ for some $L_2\ge 1$.
\end{enumerate}

Let us further assume that $X=G/\Gamma$, where $G$ is connected semisimple Lie group
without compact factors and $\Gamma$ an irreducible lattice in $G$.
Then property \eqref{eq:eq2} is the well-known exponential mixing estimate (see, for instance, \cite[2.4.3]{KM1}), where the bound involves the $L_k^2$-Sobolev norms of the test functions. 
Property \eqref{eq:eq2} has been verified in \cite{KM2,Shi,BG1} 
with respect to the norm  $S_d$ defined in \eqref{eq:S000} (see, for instance, \cite[Th. 2.2]{BG1}).

\subsection{Norms on the group $T$}
According to our assumptions on $T$ and $\mathfrak{u}:=\hbox{Lie}(U)$, there exists a subset $\Phi \subset \Hom(T,\bR^*)$ such that 
\[
\gou = \bigoplus_{\alpha \in \Phi} \gou_\alpha,
\quad\hbox{with $\gou_\alpha := \big\{ w \in \gou \, : \, \Ad(t)w = \alpha(t)w, \enskip \textrm{for all $t \in T$} \big\}$. }
\]
We choose a basis of each subspace $\gou_\alpha$. This gives a basis $\gou$.
For $w\in \gou$, we denote by $\|w\|$ the  $\ell^\infty$-norm with respect to this basis. 
Given $t \in T$, define
\begin{align}
\|t\|_* 
&:= \max \big\{ \max\big(\|\Ad(t)w\|,\|\Ad(t)^{-1}w\|\big) \, : \, \|w\| = 1 \big\} \nonumber \\[0.2cm]
&= \max \big\{ \max\big(|\alpha(t)|,|\alpha(t)|^{-1}\big) \, : \, \alpha \in \Phi \big\}. \label{idtnorm}
\end{align}
Note that $\|t^{-1}\|_*=\|t\|_*$ and $\|t\|_* \geq 1$ for all $t \in T$. Since the action of $T$ on $\gou$ is proper, $\log \|\cdot\|_*$ defines a norm on $\hbox{Lie}(T)$.
Then the invariant metric $d(t_1,t_2)$ on $T$ is comparable up to a constant to $\log \|t_1t_2^{-1}\|_*$, so that it is sufficient to establish the estimates in terms of  $\|\cdot\|_*$. From now on, we redefine $m_r$ from \eqref{defrhormr} as 
\begin{equation}\label{eq:m_r}
m_r(\underline{t}): = \min_{i \neq j} \|t_i t_j^{-1}\|_*,\quad \hbox{for  $\underline{t} = (t_1,\ldots,t_r) \in T_{+}^r$.}
\end{equation}

\section{Proof of the main theorems}
\label{Subsec:Rough}

We fix $r \geq 1$, $\varphi_o \in W(\nu)$, $\varphi_1,\ldots,\varphi_r \in \cA$ and 
$\underline{t} = (t_1,\ldots,t_r) \in T^r_+$. 
We wish to estimate the expression
\[
\left|\nu\Big(\varphi_o \prod_{i=1}^r \varphi_i \circ t_i\Big) - \nu(\varphi_o) \prod_{i=1}^r \mu(\varphi_i) \right|.
\]

\subsection{Step I: expanding the function $\varphi_o$}
\label{subsec:stepI}
\vspace{0.1cm}


By \eqref{varphiexpansion} we can write
\[
\varphi_o = \sum_{\xi \in \Xi} \nu(\varphi_o \cdot \overline{\psi}_\xi) \psi_\xi,
\]
where the series converges uniformly,  and with the convention that $\psi_1 = 1$. Then
\begin{align}
\nu\Big(\varphi_o \prod_{i=1}^r \varphi_i \circ t_i\Big) - \nu(\varphi_o) \prod_{i=1}^r \mu(\varphi_i) 
=\,&
\nu(\varphi_o) \Big( \nu\Big(\prod_{i=1}^r \varphi_i \circ t_i\Big) - \prod_{i=1}^r \mu(\varphi_i) \Big) \nonumber \\
&+
\sum_{\xi \neq 1}  \nu(\varphi_o \cdot \overline{\psi}_\xi) \, \nu_\xi\Big(\prod_{i=1}^r \varphi_i \circ t_i\Big),\label{eq:char}
\end{align}
where we used the notation
\begin{equation}
\label{Def_nu_xi}
\nu_\xi(\eta) := \nu(\psi_\xi \cdot \eta), \quad \textrm{for $\eta \in B(X)$}.
\end{equation}
Let us now define
\begin{equation}\label{eq:d1}
D_{d,r}(\underline{t};1) 
:= 
\sup\left\{ \Big|\nu\Big(\prod_{i=1}^r \varphi_i \circ t_i\Big) - \prod_{i=1}^r \mu(\varphi_i)\Big| \, : \, 
\varphi_1,\ldots,\varphi_r \in \cA, \enskip S_d(\varphi_1),\ldots ,S_d(\varphi_r)\leq 1 \right\}, 
\end{equation}
and, for $\xi \neq 1$, 
\begin{equation}\label{eq:d2}
D_{d,r}(\underline{t};\xi) 
:= 
\sup\left\{ \Big|\nu_\xi\Big(\prod_{i=1}^r \varphi_i \circ t_i\Big)\Big| \, : \, 
\varphi_1,\ldots,\varphi_r \in \cA, \enskip S_d(\varphi_1),\ldots ,S_d(\varphi_r) \leq 1 \right\}. 
\end{equation}
Finally, we set
\begin{equation}
\label{Def_Er}
E_{d,r}(\underline{t}) := \sup_{\xi \in \Xi} D_{d,r}(\underline{t};\xi).
\end{equation}
Then it follows from \eqref{eq:char} that
\begin{equation}
\label{NewEstimate}
\Big|\nu\Big(\varphi_o \prod_{i=1}^r \varphi_i \circ t_i\Big) - \nu(\varphi_o) \prod_{i=1}^r \mu(\varphi_i) \Big| 
\leq E_{d,r}(\underline{t}) \, \|\varphi_o\|_{W(\nu)} \prod_{i=1}^r S_d(\varphi_i).
\end{equation}

Our goal from now on will be to estimate the quantity $E_{d,r}(\underline{t})$.
This will be established through an elaborate induction scheme,  so it will be convenient to define
\begin{equation}
\label{Def_Er-1}
\widetilde{E}_{d,r-1}(\underline{t}) := \max\Big\{ E_{d,p}(t_{i_1},\ldots,t_{i_p}) \, : \, 1 \leq p < r, \enskip \{i_1,\ldots,i_p\} \subset \{1,\ldots,r\} \Big\}.
\end{equation}

\medskip

\subsection{Step II: An upper bound on $E_1$ (base of induction)}
\label{subsec:stepII}

In this section, we prove the Theorem \ref{ThmA} when $r=1$.
The assumption \eqref{eq:eq1} asserts that
\begin{equation}
\label{AssD_revisited}
D_{d_o,1}(t;1) \leq D_o\, \rho(t)^{-\delta_o}, \quad \textrm{for all $t \in T_{+}$}.
\end{equation} 
We aim to estimate $E_{d_1,1}(t)$ for some suitably chosen $d_1>d_o$.
In view of \eqref{Def_Er}, it suffices to bound 
\[
E'_{d_1,1}(t) := \sup_{\xi \neq 1} D_{d_1,1}(t;\xi).
\] 
To do this we shall exploit the following $U$-equivariance of the complex measures $\nu_\xi$, in combination
with the equidistribution assumption \eqref{eq:eq2}.\\

\begin{lemma}
\label{Lemma_nu_xi_is_equivariant}
For all $\xi \in \Xi$ and $w \in \gou$,
\[
\exp(w)_*\nu_\xi = \xi(-w) \nu_\xi.
\]
\end{lemma}

\smallskip

\begin{proof}
Indeed,  it follows from \eqref{psi_xi_is_an_eigenfunction} that for every $\varphi \in B(X)$, 
\begin{align*}
\exp(w)_*\nu_\xi(\varphi) 
&=
\nu_\xi(\varphi \circ \exp(w)) = \nu(\psi_\xi \cdot \varphi \circ \exp(w))
=
\nu(\psi_\xi \circ \exp(-w) \cdot \psi)\\
&= \xi(-w) \nu(\psi_\xi \cdot \varphi) 
=
\xi(-w) \nu_\xi(\varphi).
\end{align*}
\end{proof}

\noindent Let $\varphi \in \cA$, $w \in \gou \setminus \{0\}$, $t \in T_{+}$ and $\xi \neq 1$. 
Using Lemma \ref{Lemma_nu_xi_is_equivariant}, we conclude that for every $s\in\bR$,
\begin{align*}
\nu_\xi(\varphi \circ t) 
&= \xi(s w) \nu_\xi(\varphi \circ t \circ \exp( s w))
= \xi(s w) \nu_\xi(\varphi \circ \exp( s \Ad(t)w) \circ t)\\
&= \xi(s w) \nu_\xi\big(\varphi \circ \exp( s \Ad(t)w) \circ t- \mu(\phi)\big),
\end{align*}
where we used that $\nu_\xi(1) = \nu(\psi_\xi) =0$ when $\xi\ne 1$. 
For $L>0$, we set
\begin{equation}\label{eq:phiL}
\phi_{L} := \frac{1}{L} \int_0^L \xi(s w) \Big( \varphi \circ \exp(s \Ad(t)w) - \mu(\varphi) \Big) \, ds.
\end{equation}
Then
\begin{equation}
\label{fromvarphitophi}
\nu_\xi\big(\varphi \circ t \big) = \nu_\xi(\phi_{L} \circ t).
\end{equation}
Using that $|\psi_\xi| = 1$ $\nu$-almost everywhere, we deduce that
\begin{align}
|\nu_\xi(\phi_L \circ t)| 
&\leq 
\nu(|\phi_L| \circ t) \leq \nu(|\phi_L|^2 \circ t)^{1/2} \nonumber \\
&\leq 
\mu(|\phi_L|^2)^{1/2} + \big|\nu(|\phi_L|^2 \circ t) - \mu(|\phi_L|^2)\big|^{1/2}, \label{fromnutomu}
\end{align}
where we used the inequality $\alpha^{1/2}\le \beta^{1/2}+ |\alpha-\beta|^{1/2}$ with $\alpha,\beta \geq 0$.  \\

The required bounds on the two terms in \eqref{fromnutomu} 
is provided by  the following two lemmas:  \\

\smallskip

\begin{lemma}
\label{Lemma_etaL_1}
Let $c_L:= \mu(\varphi)\cdot  \frac{1}{L} \int_0^L \xi(s w) \, ds$.  
\vspace{0.1cm}
\begin{itemize}
\item[$(i)$] $\eta_L := |\phi_L|^2 - |c_L|^2 \in \cA$.  \vspace{0.2cm}
\item[$(ii)$] If $L\|\Ad(t)w\| \geq 1$, then 
$$
S_{d_o}(\eta_L) \leq B_{d_0}'  \big(L\|\Ad(t)w\|\big)^{2b_{2d_o}} S_{2d_o}(\varphi)^2,
$$
where $B_{d_0}':=M_{d_o}B_{2d_o}^2+2B_{d_o}$.
\end{itemize}
\end{lemma}

\vspace{0.2cm}

\begin{lemma} \label{Lemma_etaL_2}
If $L\|\Ad(t)w\| \geq 1$,  then
\[
\mu\big(|\phi_L|^2\big) \leq 14C \big(L\|\Ad(t)w\|\big)^{-c} \, S_{d_o}(\varphi)^2.
\]
\end{lemma}

\medskip

We postpone the proofs of the lemmas until Section \ref{sec:lemmas}
and continue with the estimate \eqref{fromnutomu}.
We note that since both $\nu$ and $\mu$ are probability measures and $c_L$ is a constant, we have
\[
\big|\nu(|\phi_L|^2 \circ t) - \mu(|\phi_L|^2)\big| 
= 
\big|\nu(\eta_L \circ t) - \mu(\eta_L)\big|.
\]
Hence, using \eqref{eq:eq1} and Lemma \ref{Lemma_etaL_1}, we deduce that
\begin{align*}
\big|\nu(|\phi_L|^2 \circ t) - \mu(|\phi_L|^2)\big| 
&\leq 
D_o\,  \rho(t)^{-\delta_o} \, S_{d_o}(\eta_L)\\
&\leq 
D_o B_{d_o}'  \big(L\|\Ad(t)w\|\big)^{2b_{2d_o}} \rho(t)^{-\delta_o} \, S_{2d_o}(\varphi)^2,
\end{align*}
for all $L> 0$ such that $L\|\Ad(t)w\| \geq 1$. 
Combining this estimate with Lemma \ref{Lemma_etaL_2}, we deduce from \eqref{fromnutomu} that
\begin{align*}
\big|\nu_\xi(\varphi \circ t)| 
=
\big|\nu_\xi(\phi_L \circ t)|
\leq  \Big( 4\sqrt{C} \big(L\|\Ad(t)w\|\big)^{-c/2} +  \sqrt{D_o B_{d_o}'} \big(L\|\Ad(t)w\|\big)^{b_{2d_o}} \rho(t)^{-\delta_o/2}\Big) S_{2d_o}(\varphi),
\end{align*}
for all $L> 0$ such that $L\|\Ad(t)w\| \geq 1$.   \\

It remains to find a suitable $L > 0$ to ensure that the right-hand side in last inequality decays like an inverse power of $\rho(t)$. If 
we choose $L > 0$ so that
\begin{equation}
\label{chooseL_1}
L \|\Ad(t)w\| = \rho(t)^{\delta_o/(c+2b_{2d_o})},
\end{equation}
then, since $\rho(t) \geq 1$, we see that $L\|\Ad(t)w\| \geq 1$, and thus the previous  bounds are available. We conclude that
\[
\big|\nu_\xi(\varphi \circ t)| \leq D'_1\, \rho(t)^{-\delta_1} \, S_{2d_o}(\varphi),
\]
where
\[
D'_1 := 5 \max\Big(\sqrt{C},\sqrt{D_o B_{d_o}'} \Big) \qand \delta_1: = \frac{c \delta_o}{2(c+2b_{2d_o})} < \delta_o < 1.
\]
Since this upper bound is uniform over all $\xi \neq 1$, and the constants are independent of $t \in T_{+}$, we obtain that
\[
E'_{2d_o,1}(t) \leq D'_1\, \rho(t)^{-\delta_1}, \quad \textrm{for all $t \in T_{+}$}.
\]
Combining this bound with \eqref{AssD_revisited}, we finally deduce that
\begin{equation}
\label{upperbndE_1}
E_{2d_o,1}(t) \leq D_1\, \rho(t)^{-\delta_1}, \quad \textrm{for all $t \in T_{+}$},
\end{equation}
where $D_1: = \max(D_o,D'_1)$.


\subsection{Step III: Choosing a suitable one-parameter subgroup}
\label{subsec:stepIII}

While the estimates in Step II involve averaging along an arbitrarily chosen one-parameter subgroup of $U$,  we will have to choose this subgroup more thoughtfully 
to handle higher order correlations. We will carry out this  task now.  \\

Let $r \geq 2$ and $\underline{t} = (t_1,\ldots,t_r) \in T_{+}^r$.
We define 
\begin{equation}
\label{Def_mrMr}
M_r(\underline{t}): = \max_{i,j} \|t_i t_j^{-1}\|_*,
\end{equation}
where $\|\cdot\|_{*}$ is given by \eqref{idtnorm}. 
By the definition of $\|\cdot\|_{*}$, we can find indices $i,j = 1,\ldots,r$ and $\alpha \in \Phi$ such that 
\[
M_r(\underline{t}) =  \|t_i t_j^{-1}\|_* = |\alpha(t_i t_j^{-1})| = \big\|\Ad(t_i t_j^{-1})e_{\alpha,1}\big\|,
\]
where $e_{\alpha,1}$ belongs to the (fixed) basis of $\gou_\alpha$.  We set $i_1 = i$, and pick indices $i_2,\ldots,i_r$, all distinct, such that
\begin{equation}
\label{order}
\big\|\Ad(t_{i_1}t_j^{-1})e_{\alpha,1}\big\| \geq \big\|\Ad(t_{i_2}t_j^{-1})e_{\alpha,1}\big\| \geq \cdots \geq \big\|\Ad(t_{i_r}t_j^{-1})e_{\alpha,1}\big\|.
\end{equation}
In particular, there exists an index $l = 1,\ldots,r$ such that $i_l = j$, and thus 
\begin{equation}
\label{order2}
1 = 
\big\|\Ad(t_{i_l} t_{j}^{-1})e_{\alpha,1}\big\| \geq \big\|\Ad(t_{i_r}t_j^{-1})e_{\alpha,1}\big\|. 
\end{equation}
Since the expression $E_{d,r}(\underline{t})$ that we wish to estimate is invariant under permutations of the elements in $\underline{t}$ we may henceforth without loss of generality
adopt the following convention:
\begin{center}
\emph{The indices are relabelled so that $i_k = k$ for $k = 1,\ldots,r$}.
\end{center}
Assuming this convention, we set
\begin{equation}
\label{Def_w}
w := \Ad(t_l^{-1})e_{\alpha,1} \qand w^{(i)}: = \Ad(t_i)w, \quad \textrm{for $i=1,\ldots,r$}.
\end{equation}
We note that 
\[
\|w^{(1)}\| = M_r(\underline{t}) \qand \|w^{(r)}\| \leq 1,
\]
and it follows from \eqref{order} and \eqref{order2} that
\begin{equation}
\label{orderw}
\|w^{(1)}\| \geq \cdots \geq \|w^{(r)}\|  \qand \|w^{(r)}\| \leq M_r(\underline{t})^{-1}\|w^{(1)}\|.
\end{equation}
We stress that the $r$-tuple $w^{(1)},\ldots, w^{(r)}$
 is uniquely determined up to permutations of indices by $\underline{t}$.

\subsection{Step IV: A recursive estimate of $E_r$ in terms of $\widetilde{E}_{r-1}$ (inductive step)}
\label{subsec:stepIV}

Let $r \geq 2$, $\underline{t} = (t_1,\ldots,t_r) \in T_{+}^r$, and $\varphi_1,\ldots,\varphi_r \in \cA$. 
For $\xi \in \Xi$, we need to estimate 
\[
\nu_{\xi}\Big(\prod_{i=1}^r \varphi_i \circ t_i\Big)= \nu \Big(\psi_\xi \prod_{i=1}^r \varphi_i \circ t_i\Big).
\]
We exploit the invariance of the measure $\nu$ under $U$
and consider the one-parameter subgroup $\exp(sw)$,  where $w$ is
determined by the tuple $\underline{t}$ and is defined in \eqref{Def_w}. 
Then
$$
\nu_\xi \Big(\prod_{i=1}^r \varphi_i \circ t_i\Big)=
\nu \Big(\psi_\xi\circ \exp(sw) \prod_{i=1}^r \varphi_i \circ t_i\circ \exp(sw) \Big)=
 \nu_\xi  \Big(\xi(sw) \prod_{i=1}^r \varphi_i\circ \exp(sw^{(i)}) \circ t_i  \Big),
$$
where $w^{(i)}: = \Ad(t_i)w$. Therefore, for any $L>0$,
\begin{equation}\label{eq:average}
\nu_\xi \Big(\prod_{i=1}^r \varphi_i \circ t_i\Big)=
\nu_\xi  \Big(\frac{1}{L}\int_0^L \xi(sw) \prod_{i=1}^r \varphi_i\circ \exp(sw^{(i)}) \circ t_i\, ds  \Big).
\end{equation}
Our argument uses induction on the number of factors $r$.
For an index $1 \leq p < r$, we set
\begin{align}
I_{1,\xi} 
&:= \nu_\xi\Big(\displaystyle \prod_{i=1}^r \varphi_i \circ t_i\Big) - \nu_\xi\left( \Big(\frac{1}{L} \int_0^L \xi(s w) \prod_{i=1}^p \varphi_i \circ \exp\big(s w^{(i)}\big) \circ t_i \, ds \Big)  \prod_{i=p+1}^r \varphi_i \circ t_i \right), \label{Def_I1_xi} \\[0.2cm]
I_{2,\xi} 
&:= \nu_{\xi}\left( \Big( \frac{1}{L} \int_0^L \xi(s w) 
\Big( \prod_{i=1}^p \varphi_i \circ \exp\big(s w^{(i)}\big) \circ t_i  - \prod_{i=1}^p \mu(\varphi_i) \Big)  \, ds \Big) \, \prod_{i=p+1}^r \varphi_i  \circ t_i \right), \label{Def_I2_xi} \\[0.2cm]
I_{3,\xi}
&:=
\left\{
\begin{array}{ll}
\displaystyle \prod_{i=1}^p \mu(\varphi_i) \, \left( \nu\Big( \displaystyle \prod_{i=p+1}^r \varphi_i \circ t_i\Big) - \displaystyle \prod_{i=p+1}^r \mu(\varphi_i) \right) 
& \textrm{if $\xi = 1$} \\[0.6cm]
\displaystyle \prod_{i=1}^p \mu(\varphi_i) \, \nu_\xi\Big( \displaystyle \prod_{i=p+1}^r \varphi_i \circ t_i\Big) & 
\textrm{if $\xi \neq 1$} 
\end{array}
\right..
\label{Def_I3_xi}
\end{align}
Then
\begin{align*}
I_{1,\xi} + I_{2,\xi} + \Big( \frac{1}{L} \int_0^L \xi(s w) \, ds \Big) I_{3,\xi}
=
\left\{
\begin{array}{cl}
\nu\Big(\displaystyle \prod_{i=1}^r \varphi_i \circ t_i\Big) - \displaystyle \prod_{i=1}^r \mu(\varphi_i) & \textrm{if $\xi = 1$} \\[0.6cm]
\nu_\xi\Big(\displaystyle \prod_{i=1}^r \varphi_i \circ t_i\Big) & \textrm{if $\xi \neq 1$}
\end{array}
\right..
\end{align*}
Since $|\xi|=1$, we see that when $S_d(\phi_1),\ldots,S_d(\phi_r)\le 1$,
\begin{equation}
\label{Alt_Er}
E_{d,r}(\underline{t}) \leq \sup_{\xi \in \Xi} |I_{1,\xi}| + \sup_{\xi \in \Xi} |I_{2,\xi}| + \sup_{\xi \in \Xi} |I_{3,\xi}|.
\end{equation}
We thus wish to prove that each term in \eqref{Alt_Er} is small, provided that $L > 0$ and the index $p \in [1,r)$ are chosen appropriately. An important ingredient towards achieving this is the following technical proposition, whose proof we postpone until Section \ref{Sec:PropMain}. 

\smallskip

\begin{proposition}
\label{Prop_mainest}
For every $L > 0$ and $1 \leq p < r$ such that $L\|w^{(p)}\| \geq 1$, 
\vspace{0.1cm}
\begin{itemize}
\item[\textsc{(I)}] $\displaystyle \sup_{\xi \in \Xi} |I_{1,\xi}| \leq rA \big(L\|w^{(p+1)}\|\big)^a \prod_{i=1}^r S_{d_o}(\varphi_i)$. \\[0.2cm]
\item[\textsc{(II)}] $\displaystyle \sup_{\xi \in \Xi} |I_{2,\xi}| \leq P_1\Big((P_d L \|w^{(1)}\|\big)^{rb_{d+d_o}} \, D_{d,p}(t_1,\ldots,t_p;1)^{1/2} 
+ \sqrt{r} \big(L\|w^{(p)}\|\big)^{-c/2} \Big) \prod_{i=1}^r S_{d+d_o}(\varphi_i)$,\\
where $P_1 := \sqrt{14C}$ and $P_d := \big(M_d B_{d+d_o}^2+2B_d^2\big)^{1/2b_{d+d_o}}$.\\[0.2cm]
\item[\textsc{(III)}] $\displaystyle \sup_{\xi \in \Xi} |I_{3,\xi}| \leq E_{d,r-p}(t_{p+1},\ldots,t_r) \, \prod_{i=1}^r S_d(\varphi_i)$ when $d\ge d_o$.
\end{itemize}
\end{proposition}

\medskip

The key point of Proposition \ref{Prop_mainest} is that it will eventually allow us to establish an upper bound $E_{d,r}(\underline{t})$ in terms of $\widetilde{E}_{d,r-1}(\underline{t})$. Using induction, we can then use our bound on $E_{d,1}$ (and thus on $\widetilde{E}_{d,1}$) from the previous steps 
to provide an upper bound on $E_{d,r}$. 


\subsection{Step V: Minimize the bound with respect to $p$ and $L$}
\label{subsec:stepV}

Here we specify the parameters $p$ and $L$ for which estimates from Step IV will be applied. The following version of the Pigeon-Hole Principle will be useful. 

\smallskip

\begin{lemma}
\label{Lemma_pigeonhole}
Fix an integer $r \geq 2$ and a real number $\theta \in (0,1)$. Then, for every $r$-tuple $(\beta_1,\ldots,\beta_r)$ of non-negative real numbers, which satisfies 
\[
\beta_r \leq \cdots \leq \beta_1 \qand \beta_r \leq \beta_1 \, \theta,
\]
there exist $1 \leq p \leq r-1$ and $0 \leq q \leq r-2$ such that
\[
\beta_{p+1} < \beta_1 \, \theta^{(q+1)/r} < \beta_1 \, \theta^{q/r}\le \beta_p.
\]
\end{lemma}

\begin{proof}
Let $\gamma_q = \beta_1 \, \theta^{q/r}$ for $0 \leq q \leq r$. Since $\theta \in (0,1)$ and $\beta_r \leq \beta_1 \theta$, we have
\[
\beta_r \leq \beta_1 \, \theta = \gamma_r < \cdots < \gamma_1 < \gamma_0 = \beta_1.
\]
and thus the $r$ distinct (and linearly ordered) points $\gamma_0,\ldots,\gamma_{r-1}$ belong to the interval $(\beta_r,\beta_1]$. 
Let us partition
this interval into $r-1$ half-open (possibly empty) intervals as
\[
(\beta_r,\beta_1] = \bigsqcup_{p=1}^{r-1} (\beta_{p+1},\beta_p].
\] 
Then, by the Pigeon-Hole Principle, there are two consecutive points which belong to same partition interval, i.e. there exist $1 \leq p \leq r-1$ and $0 \leq q \leq r-2$ such that
\[
\beta_{p+1} < \gamma_{q+1} < \gamma_q \leq \beta_p,
\]
which finishes the proof.
\end{proof}

\vspace{0.1cm}

\noindent Let $\underline{t}\in T_+^r$.
Throughout the rest of the argument, we assume that 
$$
M_r(\underline{t})=\max_{i,j} \|t_it_j^{-1}\|_* > 1,
$$
or equivalently, that $t_i \neq t_j$ for at least one pair $(i,j)$ of distinct indices. 
We shall apply Lemma \ref{Lemma_pigeonhole} above to
\[
\theta \in [M_r(\underline{t})^{-1},1) \qand \beta_i = \|w^{(i)}\|, \quad \textrm{$i=1,\ldots,r$},
\]
where the elements $w^{(i)}$ are defined in \eqref{Def_w}.
Note that \eqref{orderw} implies that the conditions in the lemma are satisfied.
We conclude that there are indices $1 \leq p < r-1$ and $0 \leq q \leq r-2$ such that
\[
\|w^{(p+1)}\| <  \|w^{(1)}\|\,\theta^{(q+1)/r} <  \|w^{(1)}\| \, \theta^{q/r} \le \|w^{(p)}\|,
\]
and we set
\[
L := \|w^{(1)}\|^{-1} \theta^{-(q+1/2)/r}.
\]
With this choice of $L$, we obtain
\begin{align}
L\|w^{(1)}\| &= \theta^{-(q+1/2)/r} \leq \theta^{-1}, \label{eq:i}\\
L\|w^{(p)}\| &= \|w^{(1)}\|^{-1} \, \theta^{-(q+1/2)/r} \|w^{(p)}\| 
\geq \theta^{-\frac{(q+1/2)}{r}} \, \theta^{\frac{q}{r}} = \theta^{-1/2r}>1, \label{eq:ii} \\
L\|w^{(p+1)}\| &
= \|w^{(1)}\|^{-1} \, \theta^{-(q+1/2)/r} \|w^{(p+1)}\| < \theta^{-(q+1/2)/r} \theta^{(q+1)/r} = \theta^{1/2r}.\label{eq:iii}
\end{align}
We stress that these bounds are independent of the indices $p$ and $q$. 

\medskip

Let us now utilize these bounds in combination with Proposition \ref{Prop_mainest}. From \eqref{eq:iii} and Proposition \ref{Prop_mainest} \textsc{(I)}, we see that
\[
\sup_{\xi \in \Xi} |I_{1,\xi}| \leq rA \theta^{a/2r} \, \prod_{i=1}^r S_{d_o}(\varphi_i).
\]
Using \eqref{eq:i}, \eqref{eq:ii}, and the trivial bound
\[
D_{d,p}(t_1,\ldots,t_p;1) \leq \widetilde{E}_{d,r-1}(\underline{t}), 
\]
we deduce from Proposition \ref{Prop_mainest}\textsc{(II)} that
\[
\sup_{\xi \in \Xi} |I_{2,\xi}| \leq P_1\Big((P_d \theta^{-1}\big)^{rb_{d+d_o}} \, \widetilde{E}_{d,r-1}(\underline{t})^{1/2} 
+ \sqrt{r} \, \theta^{c/4r} \Big) \prod_{i=1}^r S_{d+d_o}(\varphi_i).
\]
Finally, since $E_{d,r-p}(t_{p+1},\ldots,t_r) \leq \widetilde{E}_{d,r-1}(\underline{t}) \leq 2$, we conclude from Proposition \ref{Prop_mainest}\textsc{(III)} that
\[
\sup_{\xi \in \Xi} |I_{3,\xi}| \leq \widetilde{E}_{d,r-1}(\underline{t}) \prod^r_{i=1} S_d(\varphi_i) \leq \sqrt{2} \, \widetilde{E}_{d,r-1}(\underline{t})^{1/2} \prod^r_{i=1} S_d(\varphi_i)
\]
Finally, using \eqref{Alt_Er},
 we conclude that when $S_{d+d_o}(\phi_1),\ldots,S_{d+d_o}(\phi_r)\le 1$,
\begin{align}
E_{d+d_o,r}(\underline{t}) 
&\leq 
\sup_{\xi \in \Xi} |I_{1,\xi}| + \sup_{\xi \in \Xi} |I_{2,\xi}| + \sup_{\xi \in \Xi} |I_{3,\xi}| \nonumber \\
&\leq 
2 P_1(P_d \theta^{-1})^{rb_{d+d_o}} \widetilde{E}_{d,r-1}(\underline{t})^{1/2} + r \, Q\, \theta^{c_1/r}, \label{recursivebound}
\end{align}
for all $\underline{t} \in T_{+}^r$, where $Q: = 2\max(A,P_1)$ and $c_1:=\min(a/2,c/4)$. We recall that this bound  holds under the standing assumptions that 
$M_r(\underline{t}) > 1$ and $\theta \in [M_r(\underline{t})^{-1},1)$.

\subsection{Proof of Theorem \ref{ThmA}}

Given $r \geq 2$ and $\underline{t} = (t_1,\ldots,t_r) \in T_{+}^r$, we 
use the quantities $\rho_r(\underline{t})$, $m_r(\underline{t})$, and 
$\Delta_r(\underline{t})$ introduced in \eqref{defrhormr}, \eqref{Deltar}, and \eqref{eq:m_r}.
Throughout the computation, we assume that 
\begin{equation}
\label{eq:one1}
\Delta_r(\underline{t}) > 1.
\end{equation} 
We note that for any exhaustion 
\[
\{t_{i_1}\} \subset \{t_{i_1},t_{i_2}\} \subset \ldots \subset \{t_{i_1},\ldots,t_{i_r}\},
\]
where $t_{i_1},\ldots,t_{i_r}$ are distinct entries in the vector $\underline{t}$, we have
\begin{equation}
\label{DeltarExh}
\Delta_1(t_{i_1}) \geq \Delta_2(t_{i_1},t_{i_2}) \geq \ldots \geq \Delta_r(\underline{t}) > 1.
\end{equation}
In \eqref{upperbndE_1}, we proved
\[
E_{d_1,1}(t_i) \leq D_1\, \rho(t_i)^{-\delta_1}, \quad \textrm{for all $i = 1,\ldots,r$},
\]
where $d_1:=2d_o$, and $D_1\ge 1$ and $\delta_1 \in (0,1)$ are explicit constants. Hence, by \eqref{DeltarExh},
\begin{equation}
\label{indass1}
\widetilde{E}_{d_1,1}(t_{i_1},t_{i_2}) \leq D_1\, \Delta_r(\underline{t})^{-\delta_1},
\end{equation}
for all indices $i_1, i_2$, where $\widetilde{E}_{d,r}$ is defined by \eqref{Def_Er-1}. \\

We introduce the following inductive assumption:
\smallskip

\noindent{\textsc{Ind\textsubscript{$r$}:}} {\it There exist $d_{r-1}\in\bN$, $D_{r-1} \geq 1$ and $\delta_{r-1} \in (0,1)$ such that
\begin{equation}
\label{indass}
\widetilde{E}_{d_{r-1},r-1}(t_1,\ldots,t_r) \leq D_{r-1}\, \Delta_r(\underline{t})^{-\delta_{r-1}}.
\end{equation}
}
\smallskip

We note that \eqref{indass1} implies that the base of induction Ind\textsubscript{$2$} holds.
In what follows, we shall use our recursive bound \eqref{recursivebound} (for suitable $\theta$) to show that for every integer $r \geq 2$
\[
\textrm{Ind\textsubscript{$r$}} \implies \textrm{Ind\textsubscript{$r+1$}},
\]
and provide explicit estimates for constants $d_{r}$, $D_{r}$, $\delta_{r}$.

\medskip

We verify the inductive step \eqref{indass} under the assumption $\Delta_r(\underline{t})>1$.  Let
\[
\theta: = \Delta_r(\underline{t})^{-\eps_r}<1
\]
with a parameter $\eps_r \in (0,1)$, which will be specified later.
We have
\[
M_r(\underline{t})^{-1} \leq M_r(\underline{t})^{-\eps_r} \leq m_r(\underline{t})^{-\eps_r} \leq \Delta_r(\underline{t})^{-\eps_r} = \theta,
\]
and thus $\theta \in [M_r(\underline{t})^{-1},1)$, so the inequality \eqref{recursivebound} can be applied. 
Combining \eqref{indass} with \eqref{recursivebound}, we conclude that
\begin{align*}
E_{d_{r-1}+d_o, r}(\underline{t}) 
&\leq 
2 P_1(P_{d_{r-1}} \Delta_r(\underline{t})^{\eps_r})^{rb_{d_{r-1}+d_o}} \widetilde{E}_{d_{r-1},r-1}(\underline{t})^{1/2} + r \, Q \Delta_r(\underline{t})^{-\eps_r c_1/r} \\[0.2cm]
&\leq 
2 P_1(P_{d_{r-1}} \Delta_r(\underline{t})^{\eps_r} )^{rb_{d_{r-1}+d_o}} D_{r-1}^{1/2} \Delta_r(\underline{t})^{-\delta_{r-1}/2} + r \, Q \Delta_r(\underline{t})^{-\eps_r c_1/r} 
\\[0.2cm]
&\leq 
2 P_1P_{d_{r-1}}^{rb_{d_{r-1}+d_o}} D_{r-1}^{1/2}\,  \Delta_r(\underline{t})^{-(\delta_{r-1}-2\eps_r rb_{d_{r-1}+d_o})/2} + r \, Q \Delta_r(\underline{t})^{-\eps_r c_1/r}.
\end{align*}
Let us choose 
\begin{equation}
\label{epsr}
d_r:=d_{r-1}+d_o\quad\hbox{and}\quad \eps_r := \frac{\delta_{r-1}}{\frac{2c_1}{r} + 2rb_{d_{r-1}+d_o}},
\end{equation}
so that the two exponents in the last expression match.
Since $b_{d_{r-1}+d_o} > 1/2$ and $\delta_{r-1} \in (0,1)$, we have $\eps_r \in (0,1)$.
Then we obtain
\[
E_{d_r,r}(\underline{t}) \leq D_r \Delta_{r}(\underline{t})^{-\delta_r},
\]
where 
\begin{equation}
\label{Drdeltar0}
D_r := 2 P_1P_{d_{r-1}}^{rb_{d_{r-1}+d_o}} D_{r-1}^{1/2} + rQ \qand \delta_{r} := \frac{c_1\delta_{r-1}}{r(\frac{2c_1}{r} + 2rb_{d_{r-1}+d_o})} < \delta_{r-1},
\end{equation}
provided that $\Delta_r(\underline{t}) > 1$.
Finally, we note that $D_r \geq 2$ and $E_r \leq 2$ on $T_{+}^r$, so the last inequality holds 
trivially if $\Delta_r(\underline{t}) = 1$.
This completes the proof of Theorem \ref{ThmA}.

\section{Proof of Lemmas \ref{Lemma_etaL_1} and \ref{Lemma_etaL_2}}\label{sec:lemmas}

\begin{proof}[Proof of Lemma \ref{Lemma_etaL_1}]
	To prove (i),	we note that upon expanding the square, we have
	\begin{align*}
	|\phi_{L}|^2 &= \frac{1}{L^2} \int_0^L\int_0^L \xi((s_1-s_2) w) \big( \varphi \circ \exp(s_1 \Ad(t)w) - \mu(\varphi) \big) 
	\big( \overline{\varphi \circ \exp(s_2 \Ad(t)w)} - \overline{\mu(\varphi)} \big)\, ds_1ds_2\\
	&=R_{L,1} + R_{L,2}+|c_L|^2,
	\end{align*}
	where
	\[
	R_{L,1} := \frac{1}{L^2} \int_0^L \int_0^L \xi((s_1-s_2)w)\cdot \varphi \circ \exp(s_1 \Ad(t)w) \cdot \overline{\varphi \circ \exp(s_2 \Ad(t)w)} \, ds_1 ds_2, 
	\]
	and
	\[
	R_{L,2} := -2 \re\Big(\overline{c_L} \cdot \, \frac{1}{L} \int_{0}^L \xi(sw)\cdot\varphi \circ \exp(s \Ad(t)w) \, ds \Big).
	\]
	It readily follows from our assumption \eqref{S_is_convex0} that $R_{L,1}, R_{L,2} \in \cA$, so that 
	$$
	\eta_L = |\phi_{L}|^2- |c_{L}|^2=R_{L,1} + R_{L,2}\in \cA.
	$$
	This gives (i). 
	
	\medskip
	
	To prove (ii), we first note that our assumptions \eqref{S_is_convex} and \eqref{S_is_submultiplicative} imply that
	\begin{align*}
	S_{d_o}(R_{L,1}) 
	&\leq \frac{1}{L^2} \int_0^L \int_0^L S_{d_o}\big(\varphi \circ \exp(s_1 \Ad(t)w)\cdot \overline {\varphi \circ \exp(s_2 \Ad(t)w)}\big) \, ds_1 ds_2 \\
	&\leq \frac{M_{d_o}}{L^2} \int_0^L \int_0^L S_{2d_o}\big(\varphi \circ \exp(s_1 \Ad(t)w)\big) \, S_{2d_o}\big(\overline{\varphi \circ \exp(s_2 \Ad(t)w)}\big) \, ds_1 ds_2 \\
	&=
	M_{d_o} \Big( \frac{1}{L} \int_0^L S_{2d_o}\big(\varphi \circ \exp(s \Ad(t)w)\big) \, ds \Big)^2.
	\end{align*}
	Similarly, by \eqref{S_is_convex},
	\[
	S_{d_o}(R_{L,2}) \leq 2 \|\varphi\|_\infty \Big( \frac{1}{L} \int_0^L S_{d_o}\big(\varphi \circ \exp(s \Ad(t)w)\big) \, ds \Big).
	\]
	By \eqref{S_is_uniform3}, we have
	\begin{equation}
	\label{consAssB}
	S_{2d_o}\big(\varphi \circ \exp(s \Ad(t)w)\big) \leq B_{2d_o} \max\big(1,(s \|\Ad(t)w\|)^{b_{2d_o}}\big) \, S_{2d_o}(\varphi),
	\end{equation}
	for all $s \geq 0$. \\
	
	\noindent Let us now assume that $L\|\Ad(t)w\| \geq 1$. Then,
	\[
	\max(1,(s \|\Ad(t)w\|)^{b_{2d_o}}) \leq \big(L \|\Ad(t)w\|)^{b_{2d_o}} \quad \textrm{for all $s \in [0,L]$},
	\]
	so that we conclude from \eqref{consAssB} that
	\[
	\frac{1}{L} \int_0^L S_{2d_o}\big(\varphi \circ \exp(s \Ad(t)w)\big) \, ds \leq B_{2d_o} (L\|\Ad(t)w\|)^{b_{2d_o}} \, S_{2d_o}(\varphi).
	\]
	Hence,
	\[
	S_{d_o}(R_{L,1}) \leq M_{d_o	} B_{2d_o}^2 \left( L\|\Ad(t)w\|\right)^{2 b_{2d_o}} S_{2d_o}(\varphi)^2.
	\]
	Furthermore, by our assumption \eqref{S_is_uniform}, $\|\varphi\|_\infty \leq S_{d_o}(\varphi)$, and thus
	\[
	S_{d_o}(R_{L,2}) \leq 2B_{d_o} \left(L\|\Ad(t)w\|\right)^{b_{d_o}} S_{d_o}(\varphi)^2.
	\]
	We conclude that
	\begin{align*}
	S_{d_o}(\eta_L) 
	&\leq 
	S_{d_o}(R_{L,1}) + S_{d_o}(R_{L,2}) \leq \Big( M_{d_o} B_{2d_o}^2 (L\|\Ad(t)w\|)^{2 b_{2d_o}}  + 2B_{d_o} \left(L\|\Ad(t)w\|\right)^{b_{d_o}} \Big) S_{2d_o}(\varphi)^2 \\
	&\leq
	(M_{d_o} B_{2d_o}^2+2B_{d_o}) \left(L\|\Ad(t)w\|\right)^{2 b_{d_o}} \, S_{2d_o}(\varphi)^2,
	\end{align*}
	where we have used in the last inequality that $L\|\Ad(t)w\| \geq 1$. This proves (ii).
\end{proof}

\begin{proof}[Proof of Lemma \ref{Lemma_etaL_2}]
	Setting $\psi:= \varphi - \mu(\varphi)$, we obtain from \eqref{eq:phiL} that
	\[
	\mu(|\phi_L|^2) \leq \frac{1}{L^2} \int_0^L \int_0^L \big|\mu\big(\psi \circ \exp(s_1 \Ad(t)w)\cdot  \overline{\psi \circ \exp(s_2 \Ad(t)w)} \big)\big| \, ds_1 ds_2.
	\]
	Since the measure $\mu$ is $U$-invariant, 
	\begin{align*}
	\mu\big(\psi \circ \exp(s_1 \Ad(t)w)\cdot \overline{\psi \circ \exp(s_2 \Ad(t)w)} \big)
	&=
	\mu\big(\psi \circ \exp((s_1-s_2) \Ad(t)w)\cdot \overline{\psi} \big) \\
	&= 
	\mu(\varphi \circ \exp((s_1-s_2)\Ad(t)w)\cdot \overline{\varphi}) - |\mu(\varphi)|^2,
	\end{align*}
	for all $s_1,s_2 \in \bR$.  It follows from \eqref{eq:eq2} that
	\[
	\big|\mu\big(\varphi \circ \exp((s_1-s_2)\Ad(t)w)\cdot \overline{\varphi}) - |\mu(\varphi)|^2 \big|
	\leq C \max\big(1,|s_1-s_2|\|\Ad(t)w\|\big)^{-c} S_{d_o}(\varphi)^2.
	\]
	Hence, we conclude that
	\begin{align*}
	\mu(|\phi_L|^2) 
	&\leq C \Big( \frac{1}{L^2} \int_0^L \int_0^L \max\big(1,|s_1-s_2|\|\Ad(t)w\|\big)^{-c} \, ds_1 ds_2 \Big) S_{d_o}(\varphi)^2 \\
	&= C \Big( \frac{1}{R^2} \int_0^R \int_0^R \max\big(1,|s_1-s_2|\big)^{-c} \, ds_1 ds_2 \Big) S_{d_o}(\varphi)^2,
	\end{align*}
	where $R := L\|\Ad(t)w\|$. We shall use the following integral estimate,
	which proof we leave to the reader: for every $R \geq 1$ and $c \in (0,1)$, 
	\[
	\frac{1}{R^2} \int_0^R \int_0^R \max(1,|u-v|)^{-c} \, du dv \leq \frac{7R^{-c}}{1-c}.
	\]
	Then we obtain that
	\[
	\mu(|\phi_L|^2) \leq \frac{7C}{1-c} \big(L\|\Ad(t)w\|\big)^{-c} S(\varphi)^2 \leq 14C \big(L\|\Ad(t)w\|\big)^{-c} S_{d_o}(\varphi)^2,
	\]
	where we used that $c<1/2$. This proves the lemma.
\end{proof}

\section{Proof of Proposition \ref{Prop_mainest}}
\label{Sec:PropMain}

\subsection{Upper bound on $I_{1,\xi}$}

We recall that $\nu_\xi\Big(\displaystyle \prod_{i=1}^r \varphi_i \circ t_i\Big)$
can be rewritten using an additional average along a one-parameter subgroup $\exp(sw)$
as in \eqref{eq:average}. Then
\[
\nu_\xi\Big(\displaystyle \prod_{i=1}^r \varphi_i \circ t_i\Big) - \nu_\xi\Big( \Big(\frac{1}{L} \int_0^L \xi(s w) \prod_{i=1}^p \varphi_i \circ \exp\big(s w^{(i)}\big) \circ t_i \, ds \Big)  \prod_{i=p+1}^r \varphi_i \circ t_i \Big) 
\]
equals
\[
\nu_\xi\Big(\frac{1}{L} \int_0^L \xi(s w) \prod_{i=1}^p \varphi_i \circ \exp\big(s w^{(i)}\big) \circ t_i \, 
\Big( \prod_{i=p+1}^r \varphi_i \circ \exp\big(s w^{(i)}\big) \circ t_i - \prod_{i=p+1}^r \varphi_i \circ t_i \Big) ds \Big).
\]
Thus, since $|\xi| =1$ and $|\psi_\xi|=1$ almost everywhere,
\begin{equation}\label{i}
|I_{1,\xi}| \leq \Big(\prod_{i=1}^p \|\varphi_i\|_\infty\Big) \cdot \sup_{s \in [0,L]} \Big\|\prod_{i=p+1}^r \varphi_i \circ \exp\big(s w^{(i)}\big) \circ t_i - \prod_{i=p+1}^r \varphi_i \circ t_i\Big\|_\infty.
\end{equation}
Furthermore,
\begin{align*}
&\Big\|\prod_{i=p+1}^r \varphi_i \circ \exp\big(s w^{(i)}\big) \circ t_i - \prod_{i=p+1}^r \varphi_i \circ t_i\Big\|_\infty =
\Big\|\prod_{i=p+1}^r \varphi_i \circ \exp\big(s w^{(i)}\big) - \prod_{i=p+1}^r \varphi_i\Big\|_\infty\\
\le\; & \sum_{k=r}^{p+1}  
\Big\|\Big(\prod_{i=p+1}^k \varphi_i \circ \exp\big(s w^{(i)}\big)\Big) \prod_{i=k+1}^r \varphi_i - 
\Big( \prod_{i=p+1}^{k-1} \varphi_i \circ \exp\big(s w^{(i)}\big)\Big) \prod_{i=k}^r \varphi_i
\Big\|_\infty\\
\le\; & \sum_{k=r}^{p+1}  \big\|\varphi_k \circ \exp\big(s w^{(k)}\big) - \varphi_k\big\|_\infty \prod_{p+1\le i \neq k\le r } \|\varphi_i\|_\infty.
\end{align*}
From \eqref{S_is_uniform2}, we have 
\[
\big\|\varphi_k \circ \exp\big(s w^{(k)}\big) - \varphi_k\big\|_\infty \leq A\big(s \|w^{(k)}\|\big)^a S_{d_o}(\varphi_k).
\]
Therefore, using \eqref{S_is_uniform} and \eqref{orderw},  we conclude that
\[
\Big\|\prod_{i=p+1}^r \varphi_i \circ \exp\big(s w^{(i)}\big) \circ t_i - \prod_{i=p+1}^r \varphi_i \circ t_i\Big\|_\infty
\leq rA\big(s\|w^{(p+1)}\|\big)^a \prod_{i=p+1}^r S_{d_o}(\varphi_i)
\]
for all $s \geq 0$. Hence, using \eqref{S_is_uniform} one more time, we conclude from \eqref{i} that
\[
|I_{1,\xi}| \leq rA\big(L\|w^{(p+1)}\|\big)^a \prod_{i=1}^r S_{d_o}(\varphi_i).
\]
This verifies Proposition \ref{Prop_mainest}\textsc{(I)}.

\subsection{Upper bound on $I_{2,\xi}$}

This proof here is similar to the argument in Section \ref{subsec:stepII}, but is more involved. Let us introduce a function $\widetilde{\psi}_L : X^p \ra \bC$ defined by
\begin{equation}
\label{psitildeL}
\widetilde{\psi}_L(\underline{x}):
=
\frac{1}{L} \int_0^L \xi(s w) \Big( \prod_{i=1}^p \varphi_i(\exp(s w^{(i)})t_i.x_i) - \prod_{i=1}^p \mu(\varphi_i) \Big) \, ds
\end{equation}
for $\underline{x} = (x_1,\ldots,x_p) \in X^p$.
We denote by $\psi_L$ the restriction of $\widetilde{\psi}_L$ to the diagonal in $X^p$, i.e.,
\[
\psi_L(x): = \widetilde{\psi}_L(x,\ldots,x), \quad \textrm{for $x \in X$}.
\]
Furthermore, we set $\lambda: = \prod_{i=p+1}^r \varphi_i \circ t_i$, so that we can write
\[
I_{2,\xi} = \nu_\xi(\psi_L \cdot \lambda)= \nu(\psi_\xi\cdot \psi_L \cdot \lambda).
\]
Hence, by Cauchy-Schwartz inequality,
\begin{equation}
\label{FirstBndI2}
|I_{2,\xi}| \leq \nu(|\psi_L|^2)^{1/2} \|\lambda\|_\infty.
\end{equation}
Now, using that the inequality $u^{1/2} \leq |u-v|^{1/2}+v^{1/2}$ for all $u,v \geq 0$, we see that
\begin{align}
\nu(|\psi_L|^2)^{1/2} 
\leq \big|\nu(|\psi_L|^2) - \mu^{\otimes p}(|\widetilde{\psi}_L|^2)\big|^{1/2} + \mu^{\otimes p}(|\widetilde{\psi}_L|^2)^{1/2}, \label{SecondBndI2}
\end{align}
where $\mu^{\otimes p}$ denotes the product measure on $X^p$ induced from $\mu$. 

\subsubsection{Upper bound on $\big|\nu(|\psi_L|^2) - \mu^{\otimes p}(|\widetilde{\psi}_L|^2)\big|^{1/2}$}

Setting 
\[
\theta_L := \frac{1}{L} \int_0^L \xi(s w) \, ds\quad\hbox{and}\quad|c_L|^2: = |\theta_L|^2 \prod_{i=1}^p |\mu(\varphi_i)|^2,
\]
we have 
\[
|\widetilde{\psi}_L|^2 = \widetilde{R}_{1,L} + \widetilde{R}_{2,L} + \widetilde{R}_{3,L} + |c_L|^2,
\]
where 
\begin{align}
\label{DefR1}
\widetilde{R}_{1,L}(\underline{x}) &:= \frac{1}{L^2} \int_0^L \int_0^L \xi((s_1-s_2)w) \prod_{i=1}^p 
\varphi_i(\exp(s_1 w^{(i)})t_i.x_i) \overline{\varphi_i(\exp(s_2 w^{(i)})t_i.x_i)} \, ds_1 ds_2,\\
\label{DefR2}
\widetilde{R}_{2,L}(\underline{x}) &:= 
-\overline{\theta}_L \, \Big( \frac{1}{L} \int_0^L \xi(s w) \prod_{i=1}^p \varphi_i(\exp(s w^{(i)}) t_i.x_i) \, ds \Big) \prod_{i=1}^p \overline{\mu(\varphi_i)},\\
\widetilde{R}_{3,L}(\underline{x})& :=\overline{\widetilde{R}_{2,L}(\underline{x})}.
\end{align}
We denote by $R_{1,L}$, $R_{2,L}$ and $R_{3,L}$ the restrictions of $\widetilde{R}_{1,L}$, $\widetilde{R}_{2,L}$ and $\widetilde{R}_{3,L}$ respectively to the diagonal in $X^p$. Then
\[
|{\psi}_L|^2 = {R}_{1,L} + {R}_{2,L} + {R}_{3,L} + |c_L|^2,
\]
and 
\begin{equation}
\label{ThirdBndI2}
\big|\nu(|\psi_L|^2) - \mu^{\otimes p}(|\widetilde{\psi}_L|^2)\big| \leq| J_1| + |J_2| + |J_3|,
\end{equation}
where
\[
J_k := \nu(R_{k,L}) - \mu^{\otimes p}(\widetilde{R}_{k,L}), \quad \textrm{for $k=1,2,3$}.
\]
We further note that $|J_2| = |J_3|$.
The following lemma, whose proof is postponed until Section \ref{Sec:LemmaJ1J2},
provides estimates on the $J_k$'s. \\

\begin{lemma}
\label{Lemma_J1J2}
For every $L > 0$ such that $L\|w^{(1)}\| \geq 1$, 
\begin{itemize}
\item[(i)] $\displaystyle |J_1| \leq (M_{d} B_{d+d_o}^2)^p (L\|w^{(1)}\|)^{2p b_{d+d_o}} D_{d,p}(t_1,\ldots,t_p) \prod_{i=1}^p S_{d+d_o}(\varphi_i)^2$.
\item[(ii)] $\displaystyle |J_2|= |J_3| \leq B_d^p (L\|w^{(1)}\|)^{pb_d} D_{d,p}(t_1,\ldots,t_p) \prod_{i=1}^p S_d(\varphi_i)^2$.
\end{itemize}
\end{lemma}

\vspace{0.2cm}

In particular, since $B_{d}, M_S, L\|w^{(1)}\| \geq 1$, we conclude from this lemma and \eqref{ThirdBndI2} that
\begin{equation}
\label{FourthBndI2}
\big|\nu(|\psi_L|^2) - \mu^{\otimes p}(|\widetilde{\psi}_L|^2)\big|
\leq B_{d,p} (L\|w^{(1)}\|)^{2p b_{d+d_o}} D_{d,p}(t_1,\ldots,t_p) \prod_{i=1}^p S_{d+d_o}(\varphi_i)^2,
\end{equation}
where $B_{d,p}:=(M_d B_{d+d_o}^2)^p+2B^p_d$. 

\subsubsection{Upper bound on $\mu^{\otimes p}(|\widetilde{\psi}_L|^2)^{1/2}$}

We write the function $\widetilde{\psi}_L$ as 
\[
\widetilde{\psi}_L(\underline{x}) = \frac{1}{L} \int_0^L \xi(s w) \gamma_s(\underline{x}) \, ds,
\]
with 
\[
\gamma_s(\underline{x}) := \prod_{i=1}^p \varphi_i(\exp(s w^{(i)})t_i.x_i) - \prod_{i=1}^p \mu(\varphi_i).
\]
Let us set $\rho_j:=\phi_j-\mu(\phi_j)$. Then we can write 
$$
\gamma_s(\underline{x})=\sum_{j=1}^{p} \gamma_s^{(j)}(\underline{x}),
$$
where
$$
\gamma_s^{(j)}(\underline{x}):=\Big(  \prod_{i=1}^{j-1} \mu(\varphi_i)\Big)
\rho_j(\exp(s w^{(j)})t_j.x_j) \Big(\prod_{i=j+1}^p \varphi_i(\exp( s w^{(i)})t_i.x_i)\Big).
$$
Hence,
$$
\mu^{\otimes p}(|\widetilde{\psi}_L|^2) =
\sum_{j,k=1}^{p}\frac{1}{L^2} \int_0^L \int_0^L \xi((s_1-s_2) w)  \mu^{\otimes p} \big(\gamma_{s_1}^{(j)} \cdot \overline{\gamma_{s_2}^{(k)}} \big)\, ds_1ds_2.
$$
Since the measure $\mu$ is $G$-invariant, we have $\mu(\rho_j)=0$, so that for $j<k$,
\begin{align*}
\mu^{\otimes p} \big(\gamma_{s_1}^{(j)} \cdot \overline{\gamma_{s_2}^{(k)}} \big)
=&
\Big(  \prod_{i=1}^{j-1} |\mu(\varphi_i)|^2 \Big)
\mu(\rho_j)\overline{\mu(\phi_j)} 
\Big(\prod_{i=j+1}^{k-1} |\mu(\varphi_i)|^2 \Big)
\mu\big(\phi_k\circ \exp( s_1 w^{(k)})\cdot  \overline{\rho_k\circ \exp( s_2 w^{(k)})} \big) \\
&\times
\Big(  \prod_{i=k+1}^{p} \mu\big(\phi_i\circ \exp( s_1 w^{(i)}) \cdot \overline{\phi_i\circ \exp( s_2 w^{(i)})} \big)\Big)=0.
\end{align*}
Also,
\begin{align*}
\mu^{\otimes p} \big(\gamma_{s_1}^{(j)}\cdot \overline{\gamma_{s_2}^{(j)}} \big)
=&
\Big(  \prod_{i=1}^{j-1} |\mu(\varphi_i)|^2 \Big)
\mu\big(\rho_j\circ \exp( s_1 w^{(j)})\cdot  \overline{\rho_j\circ \exp( s_2 w^{(j)})} \big) \\
&\times \Big(  \prod_{i=j+1}^{p} \mu\big(\phi_i\circ \exp( s_1 w^{(i)}) \cdot \overline{\phi_i\circ \exp( s_2 w^{(i)})} \big)\Big).
\end{align*}
Therefore, we deduce that
\begin{align*}
\mu^{\otimes p}(|\widetilde{\psi}_L|^2) 
&\leq 
\sum_{j=1}^{p}\frac{1}{L^2} \int_0^L \int_0^L  
\mu^{\otimes p} \big(|\gamma_{s_1}^{(j)}|^2\big) \, ds_1ds_2
\\
&\le \sum_{j=1}^p \Big( \frac{1}{L^2} \int_0^L \int_0^L \big|\mu(\rho_j \circ \exp((s_1-s_2)w^{(j)})\cdot \overline{\rho}_j)\big| \, ds_1 ds_2 \Big) \, \prod_{i \neq j} \|\varphi_i\|_\infty^2.
\end{align*}
By arguing verbatim as in the proof of Lemma \ref{Lemma_etaL_2}, we see that
\begin{equation}
\label{Bndc}
\frac{1}{L^2} \int_0^L \int_0^L \big|\mu(\rho_j \circ \exp((s_1-s_2)w^{(j)}) \cdot\overline{\rho}_j)\big| \, ds_1 ds_2
\leq 14C\big(L\|w^{(j)}\|\big)^{-c} S_{d_o}(\varphi_j)^2,
\end{equation}
for all $j=1,\ldots,p$, provided that $L\|w^{(j)}\| \geq 1$.  \\

Let us now assume that $L\|w^{(p)}\| \geq 1$. By \eqref{orderw}, we then also have $L\|w^{(j)}\| \geq 1$ for all $j = 1,\ldots,p$, 
and thus we can use the bound \eqref{Bndc} for every $j$. We conclude, once again using \eqref{orderw}, that
\begin{equation}
\label{mupbnd}
\mu^{\otimes p}(|\widetilde{\psi}_L|^2) \leq 14p \,C \big(L\|w^{(p)}\|\big)^{-c} \prod_{i=1}^p S_{d_o}(\varphi_i)^2.
\end{equation}

\subsubsection{Combining the two bounds}

Our goal is to bound $|I_{2,\xi}|$ from above, uniformly over $\xi \in \Xi$. Recall from \eqref{FirstBndI2} and \eqref{SecondBndI2} that
\[
|I_{2,\xi}| \leq \Big(\big|\nu(|\psi_L|^2) - \mu^{\otimes p}(|\widetilde{\psi}_L|^2)\big|^{1/2} + \mu^{\otimes p}(|\widetilde{\psi}_L|^2)^{1/2}\Big) \, \|\lambda\|_\infty.
\]
By \eqref{S_is_uniform}, we have $\|\lambda\|_\infty \leq \prod_{i=p+1}^r S_{d_o}(\varphi_i)$. 

We recall that we have proved in \eqref{FourthBndI2} that
\[
\big|\nu(|\psi_L|^2) - \mu^{\otimes p}(|\widetilde{\psi}_L|^2)\big|
\leq B_{d,p} \big(L\|w^{(1)}\|\big)^{2p b_{d+d_o}} D_{d,p}(t_1,\ldots,t_p) \prod_{i=1}^p S_{d+d_o}(\varphi_i)^2,
\]
provided that $L\|w^{(1)}\| \geq 1$, and in \eqref{mupbnd} we have proved that
\[
\mu^{\otimes p}(|\widetilde{\psi}_L|^2) \leq 14p \,C \big(L\|w^{(p)}\|\big)^{-c} \prod_{i=1}^p S_{d_o}(\varphi_i)^2,
\]
provided that $L\|w^{(p)}\| \geq 1$. Hence, 
\begin{align*}
|I_{2,\xi}| 
&\leq \Big( B_{d,p}^{1/2} \big(L\|w^{(1)}\|\big)^{p b_{d+d_o}} D_{d,p}(t_1,\ldots,t_p)^{1/2} + (14pC)^{1/2} \big(L\|w^{(p)}\|\big)^{-c/2} \Big) \prod_{i=1}^r S_{d+d_o}(\varphi_i) \\
&\leq 
P_1 \Big( \big(P_d L \|w^{(1)}\|\big)^{rb_{d+d_o}} D_{d,p}(t_1,\ldots,t_p)^{1/2} +\sqrt{r} \big(L\|w^{(p)}\|\big)^{-c/2} \Big) \prod_{i=1}^r S_{d+d_o}(\varphi_i),
\end{align*}
where
\[
P_1 := \sqrt{14C} \qand P_d := \big(M_d B_{d+d_o}^2+2B_d^2\big)^{1/2b_{d+d_o}}.
\]
This completes the proof of Proposition \ref{Prop_mainest}\textsc{(II)}.

\subsection{Upper bound on $I_{3,\xi}$}

Let us first consider $\xi = 1$. 
It follows immediately from the definition of $D_{d,r-p}(t_{p+1},\ldots,t_r)$ 
in \eqref{eq:d1} and \eqref{S_is_uniform} that
\begin{align*}
|I_{3,1}| 
 \leq \Big(\prod_{i=1}^p \|\varphi_i\|_\infty\Big) \Big| \nu\Big( \displaystyle \prod_{i=p+1}^r \varphi_i \circ t_i\Big) 
- \prod_{i=p+1}^r \mu(\varphi_i) \Big| 
 \leq 
D_{d,r-p}(t_{p+1},\ldots,t_r) \prod_{i=1}^r S_d(\varphi_i)
\end{align*}
when $d\ge d_o$.
Similarly, if $\xi \neq 1$, by the definition of $D_{d,r-p}(t_{p+1},\ldots,t_r;\xi)$
in \eqref{eq:d2} and \eqref{S_is_uniform}, 
\begin{align*}
|I_{3,\xi}| 
\leq \Big(\prod_{i=1}^p \|\varphi_i\|_\infty \Big) \Big| \nu_\xi\Big(\prod_{i=p+1}^r \varphi_i \circ t_i\Big)\Big| 
\leq D_{d,r-p}(t_{p+1},\ldots,t_r;\xi) \prod_{i=1}^r S_d(\varphi_i)
\end{align*}
when $d\ge d_o$. Therefore, from \eqref{Def_Er} we conclude that
\[
\sup_{\xi \in \Xi} |I_{3,\xi}| \leq E_{d,r-p}(t_{p+1}\,\ldots,t_{r}) \, \prod_{i=1}^r S_d(\varphi_i),
\]
which verifies Proposition \ref{Prop_mainest}\textsc{(III)}.

\section{Proof of Lemma \ref{Lemma_J1J2}}
\label{Sec:LemmaJ1J2}

\begin{proof}[Proof of Lemma \ref{Lemma_J1J2}(i)]
Given $\underline{s} = (s_1,s_2) \in \bR^2$ and $x \in X$, we define
\[
\phi_i(\underline{s},x) :=  
\varphi_i(\exp(s_1 w^{(i)})x) \overline{\varphi_i(\exp(s_2 w^{(i)})x)}, \quad \textrm{for $i=1,\ldots,p$}.
\]
Then
\[
\widetilde{R}_{1,L}(\underline{x}) = \frac{1}{L^2} \int_0^L \int_0^L \xi((s_1-s_2)w) \prod_{i=1}^p \phi_i(\underline{s},t_ix_i) \, ds_1 ds_2,
\]
and
\begin{equation}
\label{J1bnd}
|J_1| = \big|\nu(R_{1,L}) - \mu^{\otimes p}(\widetilde{R}_{1,L})\big| \leq \frac{1}{L^2} 
\int_0^L \int_0^L \Big| \nu\Big(\prod_{i=1}^p \phi_i(\underline{s},\cdot)\circ t_i\Big) - \prod_{i=1}^p \mu\big(\phi_i(\underline{s},\cdot)\big) \Big| \, ds_1 ds_2.
\end{equation}
Since $\cA$ is a $G$-invariant $*$-algebra and $\varphi_i \in \cA$, we note that $\phi_i(\underline{s},\cdot) \in \cA$. Hence, by the definition of 
$D_{d,p}(t_1,\ldots,t_p)$ in \eqref{eq:d1},
\[
\Big| \nu\Big(\prod_{i=1}^p \phi_i(\underline{s},\cdot)\circ t_i\Big) - \prod_{i=1}^p \mu\big(\phi_i(\underline{s},\cdot)\big) \Big|
\leq D_{d,p}(t_1,\ldots,t_p) \prod_{i=1}^p S_d\big(\phi(\underline{s},\cdot)\big),
\]
for every $\underline{s} \in \bR^2$. By \eqref{S_is_submultiplicative}, 
\[
S_d\big(\phi_i(\underline{s},\cdot)\big) \leq M_d\, S_{d+d_o}\big(\varphi_i \circ \exp(s_1 w^{(i)})\big) \, S_{d+d_o}\big(\varphi_i \circ \exp(s_2 w^{(i)})\big),
\]
and by \eqref{S_is_uniform3},
\[
S_{d+d_o}\big(\varphi_i \circ \exp(s w^{(i)})\big) \leq B_{d+d_o}\, \max\big(1,s \|w^{(i)}\|\big)^{b_{d+d_o}}\, S_{d+d_o}(\varphi_i), \quad \textrm{for all $s \geq 0$}.
\]
Let us now assume that $L\|w^{(1)}\| \geq 1$. Then, by \eqref{orderw}, 
\[
\max(1,s \|w^{(i)}\|) \leq L\|w^{(1)}\|, \quad \textrm{for all $i = 1,\ldots,p$ and $s \in [0,L]$}.
\]
Therefore, we conclude that for all $\underline{s} \in [0,L]^2$,
\[
\Big| \nu\Big(\prod_{i=1}^p \phi_i(\underline{s},\cdot)\circ t_i\Big) - \prod_{i=1}^p \mu\big(\phi_i(\underline{s},\cdot)\big) \Big|
\leq (M_d B_{d+d_o}^2)^p \big(L\|w^{(1)}\|\big)^{2pb_{d+d_o}} \, D_{d,p}(t_1,\ldots,t_p) \prod_{i=1}^p S_{d+d_o}(\varphi_i)^2,
\]
and it now follows from \eqref{J1bnd} that
\[
|J_1| \leq (M_d B_{d+d_o}^2)^p \big(L\|w^{(1)}\|\big)^{2pb_{d+d_o}} \, D_{d,p}(t_1,\ldots,t_p) \prod_{i=1}^p S_{d+d_o}(\varphi_i)^2,
\]
provided that $L\|w^{(1)}\| \geq 1$.
\end{proof}

\begin{proof}[Proof of Lemma \ref{Lemma_J1J2}(ii)]
We recall that
\[
\widetilde{R}_{2,L}(\underline{x}) =
-\overline{\theta}_L \, \Big( \frac{1}{L} \int_0^L \xi(s w) \prod_{i=1}^p \varphi_i(\exp(s w^{(i)}) t_ix_i) \, ds \Big) \prod_{i=1}^p \overline{\mu(\varphi_i)}.
\]
Since $|\theta_L| \leq 1$ and the measure $\mu$ is $T$-invariant, we have
\begin{align*}
|J_2| &\leq \big|\nu(R_{k,L}) - \mu^{\otimes p}(\widetilde{R}_{k,L})\big|\\
&\le \Big( \frac{1}{L} \int_0^L \Big| \nu\Big(\prod_{i=1}^p \varphi_i\circ\exp(s w^{(i)}) \circ t_i\Big) - \prod_{i=1}^p \mu\big(\varphi_i \circ \exp(s w^{(i)})\big)\Big| \, ds \Big) \, \prod_{i=1}^p \|\varphi_i\|_\infty.
\end{align*}
By the definition of $D_{d,p}(t_1,\ldots,t_p)$ in \eqref{eq:d1}, we see that
\[
\Big| \nu\Big(\prod_{i=1}^p \varphi_i\circ\exp(s w^{(i)}) \circ t_i\Big) - \prod_{i=1}^p \mu\big(\varphi_i \circ \exp(s w^{(i)})\big)\Big|
\leq D_{d,p}(t_1,\ldots,t_p) \prod_{i=1}^p S_d\big(\varphi_i \circ \exp(s w^{(i)})\big),
\]
for all $s \in \bR$. By \eqref{S_is_uniform2},
\[
S_d\big(\varphi_i \circ \exp(s w^{(i)})\big) \leq B_d \max\big(1,s\|w^{(i)}\|\big)^{b_d}\, S_d(\varphi_i), \quad \textrm{for all $s \geq 0$}.
\]
Let us now assume that $L\|w^{(1)}\| \geq 1$. Then, by \eqref{orderw},
\[
\max(1,s\|w^{(i)}\|) \leq L\|w^{(1)}\|, \quad \textrm{for all $i=1,\ldots,p$ and $s \in [0,L]$}.
\]
Therefore,
\[
\Big| \nu\Big(\prod_{i=1}^p \varphi_i\circ \exp(s w^{(i)}) \circ t_i\Big) - \prod_{i=1}^p \mu\big(\varphi_i \circ \exp(s w^{(i)})\big)\Big|
\leq B_d^p\big(L \|w^{(1)}\|\big)^{pb_d} \, D_{d,p}(t_1,\ldots,t_p) \prod_{i=1}^p S_d(\varphi_i),
\]
for all $s \in [0,L]$. Now it readily follows that 
\[
|J_2| \leq B_d^p \big(L \|w^{(1)}\|\big)^{pb_d} \, D_{d,p}(t_1,\ldots,t_p) \prod_{i=1}^p S_d(\varphi_i)^2,
\]
provided that $L\|w^{(1)}\| \geq 1$.
\end{proof}

\section{Explicit version of the main theorem}\label{sec:explicit}

Let us now present a version of Theorem \ref{ThmA} with explicit parameters.
Here we additionally assume that the estimates \eqref{S_is_uniform3} and \eqref{S_is_submultiplicative} hold
for
$$
\hbox{$B_d=L_1^d$ and $b_d=\ell d$ with $L_1,\ell \ge 1$}\quad\hbox{and}\quad \hbox{$M_d=L_2^d$ with $L_2\ge 1$}
$$
respectively.
Such bounds can be verified for the Sobolev norms (see Section \ref{sec:sobolev}). Then we prove:

\smallskip

\begin{theorem}
	\label{ThmB}
	There exist $H_1, H_2 >0$  and $\lambda>1$
	such that 
	for all $r \geq 1$, $\varphi_o \in W(\nu)$, $\varphi_1,\ldots,\varphi_r \in \cA$ and $\underline{t} \in T_{+}^r$ satisfying $\Delta_r(\underline{t}) > H_2^{(r-1)!r!(r+1)! \, \lambda^r}$,
	\[
	\Big| \nu\Big(\varphi_o \prod_{i=1}^r \varphi_i \circ t_i\Big) - \nu(\varphi_o) \prod_{i=1}^r \mu(\varphi_i) \Big| \leq 
	r H_1\, \Delta_r(\underline{t})^{-\frac{1}{(r!)^2 (r+1)!\, \lambda^r}} \, \prod_{i=1}^r S_{(r+1)d_o}(\varphi_i),
	\]
\end{theorem}

\smallskip

\begin{proof}
We recall from the proof of Theorem \ref{ThmB} that the parameters $\delta_r$, $D_r$, and $\delta_r$ are determined by the following relations:
\begin{align*}
d_r &=d_{r-1}+d_o\quad \hbox{with $d_1=2d_o$},\\
D_r &= 2 P_1P_{d_{r-1}}^{rb_{d_{r-1}+d_o}} D_{r-1}^{1/2} + rQ,\quad \hbox{where $P_d := \big(M_d B_{d+d_o}^2+2B_d^2\big)^{1/2b_{d+d_o}}$,}\\
\delta_{r} &= \frac{c_1\delta_{r-1}}{r(\frac{2c_1}{r} + 2rb_{d_{r-1}+d_o})}. \end{align*}
It is clear from the recursive formulas that 
$$
d_r=(r+1)d_o
$$
and 
$$
\delta_{r} = \frac{c_1\delta_{r-1}}{r(\frac{2c_1}{r} + 2\ell r (r+1))}\ge 
 \frac{1}{(r!)^2 (r+1)!\, \lambda^{r}}
$$
for an explicit $\lambda>1$, which depends only on $a$, $b$, $c$, and $\ell$.
However, with this choice, $D_r$ grows super-exponentially fast in $r$. We modify the proof of Theorem \ref{ThmA} choosing $\theta$ differently to get rid of the $P_{d_{r-1}}^{rb_{d_{r-1}}}$-factor. This will then imply that the constant grows linearly. Unfortunately, this type of argument only applies to those $\underline{t} \in T_{+}^r$ such that $\Delta_r(\underline{t})$ is sufficiently large depending on $r$.  \\
	
	We assume that the inductive assumption \eqref{indass} holds and choose
	\[
	\theta := P_{d_{r-1}} \Delta_r(\underline{t})^{-\eps_r},
	\]
	where $\eps_r$ is given by \eqref{epsr}, and $P_{d_{r-1}}$ is defined in Proposition 
	\ref{Prop_mainest}. We note that
	$$
	P_{d_{r-1}}=\big(M_{rd_o} B_{(r+1)d_o}^2+2B_{rd_o}^2\big)^{1/(2(r+1)d_o)}=
	\big(L_2^{rd_o} L_1^{2(r+1)d_o}+2L_1^{2rd_o}\big)^{1/(2(r+1)d_o)}\le L_1(L_2+2).
	$$
	Since $P_{d_{r-1}} \geq 1$, we have $M_r(\underline{t})^{-1} \leq \theta$, and we note that
	$\theta < 1$ {provided} that 
	\begin{equation}
	\label{Deltar2}
	\Delta_r(\underline{t}) > P_{d_{r-1}}^{1/\eps_r}.
	\end{equation}
	Hence, if $\underline{t}$ satisfies \eqref{Deltar2}, then, just as in the previous proof, it follows from \eqref{recursivebound} combined with \eqref{indass} that
	\[
	E_{d+d_o,r}(\underline{t}) \leq D_r \Delta_r(\underline{t})^{-\delta_r},
	\]
	where
	\begin{equation}
	\label{recursion}
	D_r := 2 P_1 D_{r-1}^{1/2} + rQ_r \quad \hbox{with $Q_r :=Q P_{d_{r-1}}^{c_1/r}$, }
	\end{equation}
and $\delta_r$ as above.
	Then
	\[
	{1}/{\eps_r} = \frac{\frac{2c_1}{r} + 2rb_{d_{r-1}+d_o}}{\delta_{r-1}} \leq \left(\frac{2c_1}{r} + 2\ell r(r+1)\right)((r-1)!)^2 r!\lambda^{r-1} \leq \gamma\, (r-1)!r! (r+1)! \lambda^{r},
	\]
	where $\gamma:=\frac{2\max(c_1,2\ell)}{\lambda}$. By induction on $r$, one can also show
	that there exists $H_1\ge 1$ such that $D_r\le H_1 r$ 
	for all $r$. Therefore, we conclude that 
	\[
	E_r(\underline{t}) \leq rH_1 \Delta_{r}(\underline{t})^{-\frac{1}{(r!)^2(r+1)! \lambda^r}}, 
	\]
	for all $\underline{t} \in T^r$ satisfying
	\[
	\Delta_r(\underline{t}) > H_2^{(r-1)! r! (r+1)! \, \lambda^r},\quad \hbox{where $H_2: =L_1^\gamma(L_2+2)^\gamma$.}
	\]
	This finishes the proof. 
\end{proof}

In the case of the examples (i)--(ii) of norms introduced in Section \ref{sec:sobolev}, the above estimates can be simplified.
Let us now additionally assume that $X=G/\Gamma$, where $G$ is connected semisimple Lie group without compact factors and $\Gamma$ an irreducible lattice in $G$,
and $\cA=\bC+C^\infty_c(X)$.

\begin{corollary}
Let $S_d$ denote the norms defined in either (i) or (ii) in Section \ref{sec:sobolev}.
Then there exist $H_1, H_2 >0$  and $\lambda>1$
such that 
for all $r \geq 1$, $\varphi_o \in W(\nu)$, $\varphi_1,\ldots,\varphi_r \in \cA$ and $\underline{t} \in T_{+}^r$ satisfying $\Delta_r(\underline{t}) > H_2^{(r-1)!r! \, \lambda^r}$,
\[
\Big| \nu\Big(\varphi_o \prod_{i=1}^r \varphi_i \circ t_i\Big) - \nu(\varphi_o) \prod_{i=1}^r \mu(\varphi_i) \Big| \leq 
r H_1\, \Delta_r(\underline{t})^{-\frac{1}{(r!)^2 \, \lambda^r}} \, \prod_{i=1}^r S_{r+1}(\varphi_i).
\]
\end{corollary}

\begin{proof}
We note that in this case, properties (S3)--(S4) holds with fixed constants
independent of $d$. Using this we obtain that 
$$
\delta_{r} = \frac{c_1\delta_{r-1}}{r(\frac{2c_1}{r} + 2rb)}\ge  \frac{1}{(r!)^2 \, \lambda^{r}}
$$
and 
\[
{1}/{\eps_r} = \frac{\frac{2c_1}{r} + 2rb}{\delta_{r-1}} \leq \gamma\, (r-1)!r!  \lambda^{r}.
\]
This implies the estimate.
\end{proof}


\begin{thebibliography}{99}
\bibitem{BEG}
Bj\"orklund, M.; Einsiedler, M.; Gorodnik, A.
\emph{Quantitative multiple mixing.}  
J. Eur. Math. Soc. (JEMS) \textbf{22} (2020), no. 5, 1475--1529. 

\bibitem{BG1}
Bj\"orklund, M.; Gorodnik, A.
\emph{Central limit theorems for Diophantine approximants.} 
Math. Ann. \textbf{374} (2019), no. 3-4, 1371--1437. 

\bibitem{DKL}
Dabbs, K.; Kelly, M.; Li, H.
\emph{Effective equidistribution of translates of maximal horospherical measures in the space of lattices.}  
J. Mod. Dyn. \textbf{10} (2016), 229--254.

\bibitem{Dol}
 Dolgopyat, D.
 \emph{Limit theorems for partially hyperbolic systems.} Trans. Amer. Math. Soc. 356 (2004), no. 4, 1637--1689.

\bibitem{Hej}
Hejhal, Dennis A. 
\emph{On value distribution properties of automorphic functions along closed horocycles.} XVIth Rolf Nevanlinna Colloquium (Joensuu, 1995), 39--52, de Gruyter, Berlin, 1996. 

\bibitem{KM1}
Kleinbock, D. Y.; Margulis, G. A.
\emph{Bounded orbits of nonquasiunipotent flows on homogeneous spaces.}  Sinai's Moscow Seminar on Dynamical Systems, 141--172, 
Amer. Math. Soc. Transl. Ser. 2, 171, Adv. Math. Sci., \textbf{28}, Amer. Math. Soc., Providence, RI, 1996.

\bibitem{KM2}
Kleinbock, D. Y.; Margulis, G. A.
\emph{On effective equidistribution of expanding translates of certain orbits in the space of lattices.} Number theory, analysis and geometry, 
385--396, Springer, New York, 2012.

\bibitem{Mar}
Margulis, G. On some aspects of the theory of Anosov systems.
Springer Monographs in Mathematics. Springer-Verlag, Berlin, 2004.

\bibitem{Sar}
Sarnak P.,
\emph{Asymptotic behavior of periodic orbits of the horocycle flow and Eisenstein series.}
Comm. Pure Appl. Math. 34 (1981), no. 6, 719--739. 

\bibitem{Shi}
Shi, R. 
\emph{Expanding Cone and Applications to Homogeneous Dynamics}, 
IMRN, \url{https://doi.org/10.1093/imrn/rnz052}

\bibitem{Zag} Zagier, D.
\emph{Eisenstein series and the Riemann zeta function.} Automorphic forms, representation theory and arithmetic (Bombay, 1979), pp. 275--301,
Tata Inst. Fund. Res. Studies in Math., 10, Tata Inst. Fundamental Res., Bombay, 1981. 

\end{thebibliography}
\end{document}